\documentclass[reqno, 10pt]{amsart}
\usepackage[mathscr]{eucal}
\usepackage[all]{xy}
\usepackage{mathrsfs}
\usepackage{xypic}
\usepackage{amsfonts,mathtools}
\usepackage{amsmath}
\usepackage{amsthm}
\usepackage{amssymb}
\usepackage{latexsym}
\usepackage{tabularx}
\usepackage{graphicx}
\usepackage{tikz}
\usepackage{tikz-cd}
\usepackage[colorlinks=true, citecolor=blue, urlcolor=blue, linkcolor=blue]{hyperref}
\usepackage[active]{srcltx}
\usetikzlibrary{positioning}
\newtheorem{thm}{Theorem}[section]

\newtheorem{lem}[thm]{Lemma}
\newtheorem{prop}[thm]{Proposition}

\theoremstyle{definition}
\newtheorem{defn}[thm]{Definition}
\newtheorem{rem}[thm]{Remark}
\newtheorem{ex}[thm]{Example}
\newcommand{\ch}{\mathrm{ch_{BC}}}
\newcommand{\chr}{\mathrm{ch_{BC,\mathbb{Q}}}}

\newcommand{\DC}{\mathrm{D^{b}_{coh}}}
\newcommand{\CO}{\mathcal{O}}
\newcommand{\SE}{\mathscr{E}}

\newcommand{\CE}{\mathcal{E}}
\newcommand{\CF}{\mathcal{F}}
\newcommand{\CG}{\mathcal{G}}
\newcommand{\C}{\mathcal{C}^{\infty}}
\newcommand{\BC}{\hat{c}}
\allowdisplaybreaks
\numberwithin{equation}{section}
\makeatletter
\@namedef{subjclassname@2020}{%
  \textup{2020} Mathematics Subject Classification}
\makeatother

\makeatletter
\def\@tocline#1#2#3#4#5#6#7{\relax
  \ifnum #1>\c@tocdepth 
  \else
    \par \addpenalty\@secpenalty\addvspace{#2}%
    \begingroup \hyphenpenalty\@M
    \@ifempty{#4}{%
      \@tempdima\csname r@tocindent\number#1\endcsname\relax
    }{%
      \@tempdima#4\relax
    }%
    \parindent\z@ \leftskip#3\relax \advance\leftskip\@tempdima\relax
    \rightskip\@pnumwidth plus4em \parfillskip-\@pnumwidth
    #5\leavevmode\hskip-\@tempdima
      \ifcase #1
       \or\or \hskip 1em \or \hskip 2em \else \hskip 3em \fi%
      #6\nobreak\relax
    \dotfill\hbox to\@pnumwidth{\@tocpagenum{#7}}\par
    \nobreak
    \endgroup
  \fi}
\makeatother
\begin{document}

\title{Refined Chern characteristic classes of blow-ups}

\author[X. Wu]{Xiaojun Wu}
\address{Laboratoire J.-A. Dieudonn\'{e}, U.M.R. 7351, Universit\'{e} C\^{o}te d'Azur, Parc Valrose 06108 Nice Cedex 02, France}%
\email{Xiaojun.WU@univ-cotedazur.fr}%

\author[S. Yang]{Song Yang}
\address{Center for Applied Mathematics and KL-AAGDM, Tianjin University, Tianjin 300072, P. R. China}%
\email{syangmath@tju.edu.cn}%

\author[X. Yang]{Xiangdong Yang}
\address{Department of Mathematics, Lanzhou University, Lanzhou 730000, P.R. China}
\email{yangxd@lzu.edu.cn}

\date{\today}


\begin{abstract}
We prove a blow-up formula for refined Chern characteristic classes of compact complex manifolds.
To this end, we establish a version of Riemann--Roch without denominators for the refined Chern characteristic classes.
In particular, as an application, we study the behaviour of the refined Chern characteristic classes of the Iwasawa manifold under a blow-up transformation.
\end{abstract}

\subjclass[2020]{32Q55; 14E05; 14C17}

\keywords{Bott--Chern cohomology, Characteristic classes, Blow-ups}
\maketitle
\setcounter{tocdepth}{1}
\tableofcontents
\section{Introduction}
In algebraic geometry, one of the most important geometric operations is the blow-up which plays a significant role in birational geometry.
From the viewpoint of birational geometry, it is a natural problem to compare the Chern classes or other characteristic classes of an algebraic variety under a blow-up transformation.
This problem was first propounded by Todd  \cite{Tod38} and further studied by Todd \cite{Tod38a, Tod41, Tod57},  Segre \cite{Seg54} and van de Ven \cite{van56}.
Here the Chern classes of an algebraic variety are understood as elements in the Chow group of the algebraic variety.
For non-singular irreducible projective algebraic varieties over an algebraically closed field of characteristic zero,
the general blow-up formula of Chern classes was proved by Porteous \cite{Por60} using Riemann--Roch theorem--Grothendieck (RRG), and the proof by Porteous only valid modulo torsion in positive characteristic until Riemann--Roch without denominators was proved.
In \cite{LS75, LS78}, Lascu--Scott gave an alternative proof for the blow-up formula of Chern classes without using RRG.
This classical blow-up formula for Chern classes has been generalized to the cases of complex manifolds by Atiyah--Hirzebruch \cite{AH62}, symplectic manifolds by Geiges--Pasquotto \cite{GP07}, blow-ups of singular varieties with regularly embedding centers by Aluffi \cite{Alu10}, and weighted blow-ups by  Mus\c{t}at\v{a}--Mus\c{t}at\v{a} \cite{MM12} and  Abramovich--Arena--Obinna \cite{ASA23}, etc.

Consider a compact complex manifold $Y$ with complex dimension $n$.
Then we have a natural double complex $(\Omega^{\bullet,\bullet}(Y,\mathbb{C}), \partial^{Y}, \overline{\partial}^{Y})$ called the Dolbeault double complex of $Y$.
Let $E$ be a holomorphic vector bundle over $Y$.
In 1965, Bott--Chern \cite{BC65} proved a double transgression formula which says that the Chern forms $c(E, h)$ and $c(E, h^{\prime})$ corresponding to different Hermitian metrics $h$ and $h^{\prime}$ satisfy the equation
$$c(E, h)-c(E, h^{\prime})=\partial^{Y}\overline{\partial}^{Y}\lambda$$
for some $\lambda$.
This leads to the definition of refined Chern characteristic classes (or Chern classes in the Bott--Chern cohomology) of the holomorphic vector bundle $E$ which lies in the quotient space
$$
H^{\bullet,\bullet}_{BC}(Y):=\frac{\ker\partial^{Y}\,\cap\,\ker\overline{\partial}^{Y}}{\mathrm{im}\partial^{Y}\overline{\partial}^{Y}}.
$$
Now we call $H^{\bullet, \bullet}_{BC}(Y)$ the Bott--Chern cohomology of $Y$.
It is worth noting that the Bott--Chern cohomology refines the de Rham cohomology.
If $Y$ is a compact K\"{a}hler manifold or more generally satisfies the $\partial^{Y}\overline{\partial}^{Y}$-Lemma, then the Bott--Chern cohomology of $Y$ is canonically isomorphic to its Dolbeault cohomology.
In general,  this is not the case and Bott--Chern cohomology has become a very useful tool in complex geometry, see \cite{Sch07, AT13, Bis13, Qia16, RYY19, RYY20, CGL20, YY20, Ste21, Ste21b, Wu23, YY23, BR23} etc.

Akin to birational geometry, in complex geometry, we consider the bimeromorphic transformations of compact complex manifolds.
A blow-up of a given complex manifold with a smooth center is the simplest form of a bimeromorphic transformation.
Particularly, on account of the weak factorization theorem \cite{AKMW02}, each bimeromorphic map between compact complex manifolds is the composition of finite blow-ups and blow-downs with smooth centers.
From the viewpoint of bimeromorphic geometry, it is natural to study the behavior of the refined Chern characteristic classes of a compact complex manifold under a blow-up.
Based on the RRG for coherent sheaves on complex manifolds attributed to Bismut--Shen--Wei \cite{BSW23}, we prove the Riemann--Roch without denominators (Theorem \ref{RR-domi}) and then we establish the blow-up formula of refined Chern characteristic classes.

\begin{thm}\label{thm1}
Assume that $Y$ is a compact complex manifold of complex dimension $n$.
Let $\imath:X\hookrightarrow Y$ be a closed complex submanifold of the complex codimension $r\geq2$ and let $\pi:\widetilde{Y}\rightarrow Y$ be the blow-up of $Y$ with the center $X$ and the exceptional divisor $E$.
Denote by $N$ the normal bundle of $X$ in $Y$ and $\zeta$ the first refined Chern characteristic class of the line bundle $\CO_{E}(1)$.
Then the refined Chern characteristic classes satisfies the identity
\begin{equation*}
\hat{c}(\widetilde{Y})-\pi^{\ast}\hat{c}(Y)
=\jmath_{\ast}
\bigl( \rho^{\ast}\hat{c}(X)\cdot\alpha\bigr),
\end{equation*}
where $\jmath:E\hookrightarrow \widetilde{Y}$ is the inclusion, $\rho=\pi|_{E}:E\rightarrow X$, and
$$\alpha= \frac{1}{\zeta}\cdot\biggl[\sum^{r}_{i=0}\rho^{\ast}\hat{c}_{r-i}(N)
-(1-\zeta)\sum^{r}_{i=0}(1+\zeta)^{i}\rho^{\ast}\hat{c}_{r-i}(N)\biggr].
$$
\end{thm}

To establish the formula above, we need some good functorial properties of the characteristic classes.
For this, it is necessary to pass to the derived category of bounded complexes of $\CO_{X}$-modules with coherent cohomology.
Following \cite{BSW23}, the underlying idea goes back to Block's notion of antiholomorphic superconnections introduced in \cite{Blo10}.
As explained in \cite{Pal06}, this idea may be interpreted as follows.
The classical Koszul--Malgrange theorem \cite{KM58} shows that if a differentiable complex vector bundle over a complex manifold admits a connection of $(0,1)$-type satisfying the integrability condition, then it has the structure of a holomorphic vector bundle.
This correspondence can be extended to coherent analytic sheaves by considering the sheaf over the ring of smooth functions satisfying appropriate local finiteness conditions together with a compatible connection.
This construction yields a category equivalent to that of coherent analytic sheaves, and hence the same derived category which enables one to define the Chern characters explicitly using the Chern--Weil theory of superconnections.
We emphasise that, despite the isomorphism between a coherent analytic sheaf and a superconnection is not canonical under the equivalence of the derived categories, the cohomology class represented by the Chern--Weil form of the superconnection in the real Bott--Chern cohomology depends only on its isomorphism class.
This is reason why the Chern character (resp. Chern classes) of a coherent analytic sheaf is well-defined in the real Bott--Chern cohomology.

In order to generalize the Riemann--Roch--Grothendieck theorem to coherent analytic sheaves on compact complex manifolds, over the past decades, many mathematicians have devoted considerable efforts to defining Chern characters on the Grothendieck group of coherent sheaves which take values in various cohomologies, see \cite[\S 1.3]{BSW23}.
In \cite{Sch07}, analogous to Deligne cohomology, Schweitzer introduced a new hypercohomology description of the Bott--Chern cohomology and defined the integral (resp. rational) Bott--Chern cohomology of complex manifolds.
In contrast to the real case, we can not carry the Chern--Weil theory of superconnections to define the Chern character in rational Bott--Chern cohomology.
Recently, by the axiomatic approach of Grivaux \cite{Gri10}, the first named author defined the Chern classes of coherent sheaves in the rational Bott--Chern cohomology (cf. \cite{Wu23}).
With the same setting as in Theorem \ref{thm1}, it is natural to ask whether the blow-up formula for Chern characteristic classes in the rational Bott--Chern cohomology is valid, to which we give an affirmative answer (Theorem \ref{thm2}).

From the analytic point of view, particularly, in the study of (singular) Hermitian-Einstein metrics on stable torsion-free coherent sheaves, one aims to study the Chern curvature (and hence the Chern classes) of a metric that is smooth on the locally free locus of the sheaf which satisfies suitable admissibility conditions, such as an $L^2$ condition, along the non-locally-free locus.
There are many results concerning Chern classes of degree up to the codimension of the non-locally-free locus.
We refer to the fundamental paper \cite{BS94} as an excellent illustration, which investigates the first and second Chern classes of a reflexive sheaf.
Roughly speaking, when the codimension of the non-locally-free locus is sufficiently high, it suffices to study the locally free locus in order to understand the low-degree Chern classes.
By contrast, our approach is designed to construct the refined Chern characteristic classes uniformly, without requiring any codimension assumption on the non-locally-free locus.

The paper is organized as follows.
In Section \ref{sec2}, we review the definition of Chern character of coherent sheaves on complex manifolds.
In Section \ref{sec3}, we describe the Chern characteristic classes of coherent sheaves.
In Section \ref{sec4}, we give the proofs of the Riemann--Roch without denominators and Theorem \ref{thm1}.
In Section \ref{sec5}, we show that the blow-up formula for Chern characteristic classes in the rational Bott--Chern cohomology holds.
In Section \ref{sec6}, we compute the  refined Chern characteristic classes of blowing up complex surfaces and complex threefolds.
In Appendix \ref{app}, we collect some key formulae on refined Chern characteristic classes.

\subsection*{Acknowledgement}
Most of this work was completed during the third named author's visit at the Chern Institute of Mathematics (CIM) in May and June 2024.
He would like to thank Professor Huitao Feng for hosting his visit and sincerely thank the CIM for the hospitality and the excellent working environment.
This work is supported by the \lq\lq Excellence Fellowships for Young Researchers" provided by IdEx of Universit\'{e} C\^{o}te d'Azur, the Visiting Scholars Program of CIM and the National Nature Science Foundation of China (Grant Numbers:12171351, 12271225).
Finally, we thank the referee for a careful reading and many useful suggestions.
\section{Chern characters of coherent sheaves}\label{sec2}

In this section, we review the construction of Chern character for coherent sheaves on complex manifolds, which are understood as elements in real Bott--Chern cohomology.
We follow the notations in \cite{BSW23}, see also \cite{Qia16, BR23}.
Throughout of this section we assume that $X$ is a compact complex manifold of complex dimension $n$.
For simplicity, we fix some notations for later use:
\begin{itemize}
\item[--] $\mathscr{A}\widehat{\otimes}\mathscr{A^{\prime}}$ = the tensor product of $\mathbb{Z}_{2}$-graded algebras $\mathscr{A}$ and $\mathscr{A}^{\prime}$ which is also $\mathbb{Z}_{2}$-graded.

\item[--] $\overline{T^{\ast}X}$ = the antiholomorphic cotangent bundle of $X$;

\item[--] $\CO_{X}$ = the structure sheaf of $X$;

\item[--] $\Omega^{p,q}(X,\mathbb{C})$ = the space of complex differential $(p,q)$-forms;

\item[--] $\Omega^{p,p}(X,\mathbb{R})$ = the space of real differential $(p,p)$-forms;

\item[--] $\Omega^{(=)}(X,\mathbb{R})=\bigoplus^{n}_{p=0}\Omega^{p,p}(X,\mathbb{R})$;

\item[--] $\mathfrak{D}^{(=)}(X,\mathbb{R})$ = the space of real currents on $X$;

\item[--] $[Z]$ = the current of integration over a complex submanifold $Z\subset X$;

\item[--] $\DC(X)$ = the derived category of bounded complexes of $\CO_{X}$-modules which have coherent cohomology.
\end{itemize}
The real Bott--Chern cohomology of $X$ is defined to be
$$
H^{(=)}_{BC}(X,\mathbb{R})=\bigoplus_{0\leq p\leq n}H^{p,p}_{BC}(X,\mathbb{R}),
$$
where
$$
H^{p,p}_{BC}(X,\mathbb{R})=\frac{\Omega^{p,p}(X,\mathbb{R})\cap\ker\,d^{X}}{\Omega^{p,p}(X,\mathbb{R})\cap\mathrm{Im}\partial^{X}\overline{\partial}^{X}}.
$$

We begin with the definition of antiholomorphic superconnections introduced by Block \cite{Blo10}, and we follow the notations in \cite[Chapter 5]{BSW23}.
For any integers $r^{\prime}$ and $r$ with $r^{\prime}<r$, let $E=\bigoplus^{r}_{i=r^{\prime}}E^{i}$ be a complex $\mathbb{Z}$-graded vector bundle on $X$.
Assume that $E$ is a left $\wedge(\overline{T^{\ast}X})$-module and hence each antiholomorphic $p$-form $\alpha\in\wedge^{p}(\overline{T^{\ast}X})$ acts on $E$ as a morphism of degree $p$.
For every $0\leq p\leq n$, set
$F^{p}E=\wedge^{p}(\overline{T^{\ast}X})E$.
Then we get a decreasing filtration
\begin{equation*}
F^{0}E=E\supset F^{1}E\supset\cdots\supset F^{n}E\supset F^{n+1}E=0.
\end{equation*}
Set $E^{p,q}_{0}=F^{p}E^{p+q}/F^{p+1}E^{p+q}$ and $D^{q}=E^{q}/F^{1}E^{q}$ for any $r^{\prime}\leq q\leq r$.
As a result, there exist a non-canonical splitting of vector bundles
\begin{equation}\label{no-ca-spl}
E^{q}=D^{q}\oplus F^{1}E^{q}
\end{equation}
and a non-canonical isomorphism of $\mathbb{Z}$-graded $\wedge(\overline{T^{\ast}X})$-modules
\begin{equation}\label{no-ca-iso}
E\cong E_{0}.
\end{equation}
We call the $\mathbb{Z}$-graded vector bundle $D=\bigoplus^{r}_{q=r^{\prime}}D^{q}$ the \emph{diagonal vector bundle} associated with $E$.
For any $r^{\prime}\leq i\leq r$, let $E^{i}_{0}=\bigoplus_{p+q=i}E^{p,q}_{0}$ and $E_{0}=\bigoplus^{r}_{i=r^{\prime}}E^{i}_{0}$.
It is noteworthy that, as $\mathbb{Z}$-graded filtered vector bundles over $X$, $E_{0}$ and $D$ has the same properties as $E$.
A \emph{generalized Hermitian metric} on $D$ is defined to be a smooth section $h=\sum^{2n}_{i=0}h_{i}$ of the $\mathbb{Z}$-graded vector bundle  $\wedge(T^{\ast}_{\mathbb{C}}X)\otimes\mathrm{Hom}(D, \overline{D}^{\ast})$ which is of degree 0 and self-adjoint and satisfies the condition that $h_{0}$ is a $\mathbb{Z}$-graded Hermitian metric on $D$ (cf. \cite[Definition 4.4.1]{BSW23}).

From definition, the space of smooth sections of $E$, denoted by $\C(X, E)$, is a $\mathbb{Z}$-graded vector space over $\mathbb{C}$.
By an \emph{antiholomorphic flat superconnection} on $E$, we mean a differential operator $A^{E\prime\prime}$ acting on $\C(X, E)$ of degree 1 such that:
\begin{itemize}
\item[(i)] $A^{E\prime\prime}(\alpha s)=(\overline{\partial}^{X}\alpha)s+(-1)^{\mathrm{deg}\alpha}\alpha A^{E\prime\prime} s$, for any $\alpha\in\Omega^{0,\ast}(X, \mathbb{C})$ and $s\in\C(X,E)$;
\item[(ii)] the identity $A^{E\prime\prime}\circ A^{E\prime\prime}=0$ holds.
\end{itemize}
For the sake of simplicity, we call $\mathscr{E}=(E, A^{E\prime\prime})$ an antiholomorphic superconnection on $X$.
We can define the morphism between two different antiholomorphic superconnections on $X$ and its cone, the pullback of an antiholomorphic superconnection on $X$ with respect to a holomorphic map $f:Y\rightarrow X$, and the tensor product of antiholomorphic superconnections (cf. \cite[\S 5.6-\S 5.8]{BSW23}).
Particularly, all antiholomorphic superconnections on $X$ form a differential graded category, denoted by $\mathrm{B_{dg}}(X)$ (cf. \cite[\S 6.2]{BSW23}).

Under the splitting \eqref{no-ca-spl} and the isomorphism \eqref{no-ca-iso}, $A^{E\prime\prime}$ can be considered as an antiholomorphic superconnection $A^{E_{0}\prime\prime}$ on $E_{0}$ and we obtain an isomorphism of $\mathbb{Z}$-graded vector bundles
$$
\wedge(T^{\ast}X)\widehat{\otimes} E_{0}\cong\wedge(T^{\ast}X)\widehat{\otimes} E\cong\wedge(T^{\ast}_{\mathbb{C}}X)\widehat{\otimes} D.
$$
Denote by $\Omega(X, D)$ the space of smooth sections of $ \wedge(T^{\ast}_{\mathbb{C}}X)\widehat{\otimes} D$ and then we get an identity
$$
\Omega(X, D)=\C(X, \wedge(T^{\ast}X)\widehat{\otimes} E).
$$
The tensor product of the trivial antiholomorphic superconnection $\overline{\partial}^{X}$ on $\wedge(T^{\ast}_{\mathbb{C}}X)$ and $A^{E_{0}\prime\prime}$ gives rise to an antiholomorphic superconnection on $\wedge(T^{\ast}X)\widehat{\otimes} E_{0}$ and we still write it as $A^{E_{0}\prime\prime}$ for the simplicity.
Let $h$ be a generalized Hermitian metric on $D$.
The metric $h$ induces a non-degenerate Hermitian form $\theta_{h}$ on $\Omega(X, D)$ (cf. \cite[Definition 7.1.3]{BSW23}).
Let $A^{E_{0}\prime}$ be the formal adjoint of $A^{E_{0}\prime\prime}$ with respect to $\theta_{h}$.
Then $A^{E_{0}}=A^{E_{0}\prime}+A^{E_{0}\prime\prime}$ is an antiholomorphic superconnection on $E_{0}$ and the curvature of the superconnection $A^{E_{0}}$ is given by
$A^{E_{0},2}=[A^{E_{0}\prime}, A^{E_{0}\prime\prime}]$ which is a smooth section of the super vector bundle $\wedge(T^{\ast}_{\mathbb{C}}X)\widehat{\otimes}\mathrm{End}(D)$.
Recall the definition of supertrace.
Let $V$ be a superspace with the involution $\tau:V\rightarrow V$ defining its $\mathbb{Z}_{2}$-grading
 $V=V_{+}\oplus V_{-}$.
For any $A\in\mathrm{End}(V)$, the supertrace of $A$ is defined to be $\mathrm{Tr_{s}}[A]=\mathrm{Tr}[\tau A]$.
Note that $\wedge(T^{\ast}_{\mathbb{C}}X)\widehat{\otimes}\mathrm{End}(D)$ is a super vector bundle over $X$.
We can extend the supertrace $\mathrm{Tr_{s}}$ to a  morphism
$$
\mathrm{Tr_{s}}:\wedge(T^{\ast}_{\mathbb{C}}X)\widehat{\otimes}\mathrm{End}(D)\longrightarrow \wedge(T^{\ast}_{\mathbb{C}}X).
$$
In particular, we have the following definition of Chern character form (cf. \cite[Definition 8.1.1]{BSW23}).
\begin{defn}
The differential form
\begin{equation}\label{c-t-z-form}
\mathrm{ch}(A^{E_{0},2}, h)=
\mathrm{Tr_{s}}
\biggl[\exp\biggl(\frac{\sqrt{-1}}{2\pi}A^{E_{0},2}\biggr)\biggr]
\end{equation}
is called the \emph{Chern character form} corresponding to the generalized metric $h$.
\end{defn}

In conclusion, the Chern character forms have the following properties (cf. \cite[Theorem 8.1.2 and Theorem 8.8.1]{BSW23}).
\begin{thm}\label{chern-cat}
The Chern character form $\mathrm{ch}(A^{E_{0},2}, h)$ satisfies the conditions:
\begin{itemize}
  \item [(i)] it is a $d^{X}$-closed form which lies in $\Omega^{(=)}(X,\mathbb{R})$;
  \item [(ii)] if $\underline{h}$ is another metric on $D$ and $\underline{A}^{E_{0},2}$ its curvature, then there exists a differential form $\gamma\in\Omega^{(=)}(X,\mathbb{R})$ such that
      $$
      \mathrm{ch}(A^{E_{0},2}, h)-
      \mathrm{ch}(\underline{A}^{E_{0},2}, \underline{h})=\partial^{X}\overline{\partial}^{X}\gamma;
      $$
  \item [(iii)] the cohomology class of the Chern character form
     $$
      \mathrm{ch_{BC}}(A^{E\prime\prime}):=\{\mathrm{ch}(A^{E_{0},2}, h)\}\in H^{(=)}_{BC}(X, \mathbb{R})
      $$
      depends only on $A^{E\prime\prime}$;
  \item [(iv)] the class $\mathrm{ch_{BC}}(A^{E\prime\prime})$ is independent with the choice of the splitting \eqref{no-ca-spl} and depends only on the isomorphism class of $\mathscr{E}$ in $\mathrm{D^{b}_{coh}}(X)$.
\end{itemize}
\end{thm}

Let $\mathcal{F}^{\bullet}$ be an object in the derived category $\mathrm{D^{b}_{coh}}(X)$.
Then there is an antiholomorphic superconnection $\mathscr{E}=(E, A^{E\prime\prime})$ such that $\mathcal{F}^{\bullet}$ is isomorphic to $\mathscr{E}$ in $\mathrm{D^{b}_{coh}}(X)$ (cf. \cite[Lemma 4.6]{Blo10} or \cite[Theorem 6.3.6]{BSW23}).
Due to (iv) in Theorem \ref{chern-cat}, the following definition is well-defined.
\begin{defn}
The \emph{Chern character} of  $\mathcal{F}^{\bullet}\in\mathrm{D^{b}_{coh}}(X)$ is defined to be
 $$
\mathrm{ch_{BC}}(\mathcal{F}^{\bullet}):=\mathrm{ch_{BC}}(A^{E\prime\prime})\in H^{(=)}_{BC}(X, \mathbb{R}).
$$
\end{defn}

If $X$ is a complex projective manifold, each coherent $\CO_{X}$-modules $\CF$ admits a finite locally free resolution $F^{\bullet}\rightarrow\CF$.
This implies that the Grothendieck group of coherent $\CO_{X}$-modules is equal to the Grothendieck group of holomorphic vector bundles over $X$.
For any compact complex surface, the resolution as above also exists, see \cite{Sch82}.
However, if $X$ is not projective such resolution may not exist when $\mathrm{dim}_{\mathbb{C}}\,X\geq3$, see \cite{Voi02}.
This is the reason why the definition of characteristic classes for coherent sheaves on general complex manifolds is more complicated, see \cite{Gri10, Wu20}.

\section{Chern characteristic classes of  coherent sheaves}\label{sec3}

Based on the definition of refined Chern classes of holomorphic vector bundles \cite{BC65} and the definition of Chern characters of coherent sheaves \cite{Qia16, BSW23, BR23},
there is a notion of Chern characteristic classes of coherent sheaves.
Without claiming any originality, we collect some basic properties, such as the Whitney formula and the naturality of Chern characteristic classes, which can be deduced from the properties of the Chern character for coherent sheaves.

Let  $\mathscr{E}=(E, A^{E\prime\prime})$ be an antiholomorphic superconnection on $X$ with the diagonal vector bundle $D$.
Given a generalized Hermitian metric $h$ on $D$, using the Chern--Weil theory \cite[\S\,1.6]{Zha01}, we can define the total Chern form as follows.

\begin{defn}\label{c-form}
The \emph{total Chern form} of the antiholomorphic superconnection $\mathscr{E}=(E, A^{E\prime\prime})$ associated to the metric $h$ is defined to be
\begin{equation}\label{b-c-form}
c(\mathscr{E}, h)
=\exp\biggl(\mathrm{Tr_{s}}
\biggl[\log\biggl(I+\frac{\sqrt{-1}}{2\pi}A^{E_{0,2}}\biggr)\biggr]\biggr).
\end{equation}
\end{defn}
Consider the Chern character form
$$
\mathrm{ch}(A^{E_{0},2}, h)=
\mathrm{Tr_{s}}
\biggl[\exp\biggl(\frac{\sqrt{-1}}{2\pi}A^{E_{0},2}\biggr)\biggr].
$$
Recall the power series expansion formulae for $\exp(x)$ and $\log(1+x)$,
\begin{equation}\label{log}
\log(1+x)=x-\frac{x^{2}}{2}+\cdots+\frac{(-1)^{n+1}x^{n}}{n}+\cdots
\end{equation}
and
\begin{equation}\label{exp}
\exp(x)=1+x+\frac{x^{2}}{2!}+\cdots+\frac{x^{n}}{n!}+\cdots.
\end{equation}
According to Theorem \ref{chern-cat}, we have the following unique expression
$$
\mathrm{ch}(A^{E_{0},2}, h)=
\mathrm{Tr_{s}}
\biggl[\exp\biggl(\frac{\sqrt{-1}}{2\pi}A^{E_{0},2}\biggr)\biggr]
=\sum^{n}_{i=0}\omega_{i},
$$
where $\omega_{i}\in\Omega^{i,i}(X,\mathbb{R})$ is $d^{X}$-closed and $\omega_{0}$ is the rank of $E$.
Comparing \eqref{log} and \eqref{exp} derives the following expression:
\begin{equation}\label{Theta}
\Theta:=\mathrm{Tr_{s}}
\biggl[\log\biggl(I+\frac{\sqrt{-1}}{2\pi}A^{E_{0,2}}\biggr)\biggr]
=\sum^{n}_{i=1}(-1)^{i+1}(i-1)!\omega_{i}.
\end{equation}

In particular, based on Theorem \ref{chern-cat} we have
\begin{prop}\label{c-class}
The total Chern form $c(\mathscr{E}, h)$ satisfies the following conditions:
\begin{itemize}
  \item [(i)] the form $c(\mathscr{E}, h)$ lies in $\Omega^{(=)}(X,\mathbb{R})$ and is $d^{X}$-closed;
  \item [(ii)] let $\underline{h}$ be another generalized metric on $D$, then the difference
      $$c(\mathscr{E}, h)-c(\mathscr{E}, \underline{h})$$
      is $\partial^{X}\overline{\partial}^{X}$-exact.
\end{itemize}
\end{prop}
\begin{proof}
Observe that $\Theta$ lies in $\Omega^{(=)}(X,\mathbb{R})$ and $d^{X}\Theta=0$.
In view of \eqref{b-c-form} and \eqref{Theta}, we have $c(\mathscr{E}, h)=\exp(\Theta)$ and this implies that the form $c(\mathscr{E}, h)$
is a $d^{X}$-closed real form lying in $\Omega^{(=)}(X,\mathbb{R})$.
We now consider the assertion (ii).
Let $\underline{A}^{E_{0},2}$ be the curvature corresponding to the metric $\underline{h}$.
Then the Chern character form associated to the metric $\underline{h}$ has the expression
$$
\mathrm{ch}(\underline{A}^{E_{0},2}, \underline{h})=\sum^{n}_{i=0}\underline{\omega}_{i}
\in\Omega^{(=)}(X,\mathbb{R}).
$$
From definition, we obtain
\begin{eqnarray*}
  c(\mathscr{E}, h)-c(\mathscr{E}, \underline{h})
   &=&  \exp(\Theta)-\exp(\underline{\Theta})= \sum_{k\geq1}\frac{1}{k!}(\Theta^{k}-\underline{\Theta}^{ k}),
\end{eqnarray*}
where
$$
\underline{\Theta}
=\sum^{n}_{i=1}(-1)^{i+1}(i-1)!\underline{\omega}_{i}\in\Omega^{(=)}(X,\mathbb{R}).
$$
Owing to (ii) in Theorem\ref{chern-cat}, we have
$$
\mathrm{ch}(A^{E_{0},2}, h)-
      \mathrm{ch}(\underline{A}^{E_{0},2}, \underline{h})
=\sum^{n}_{i=1}(\omega_{i}-\underline{\omega}_{i})
=\sum^{n}_{i=1}\partial^{X}\overline{\partial}^{X}\xi_{i-1}
$$
for some smooth forms $\xi_{i-1}\in\Omega^{i-1,i-1}(X,\mathbb{R})$.
Set $\eta_{i-1}=(-1)^{i+1}(i-1)!\xi_{i-1}$, then the following equality is valid
$$
\Theta-\underline{\Theta}
=\sum^{n}_{i=1}\partial^{X}\overline{\partial}^{X}\eta_{i-1}.
$$
Note that both $\Theta$ and $\underline{\Theta}$ are $d^{X}$-closed forms lying in $\Omega^{(=)}(X, \mathbb{R})$.
Consequently, the form $\Theta$ (resp. $\underline{\Theta}$) is  $\partial^{X}$ and $\overline{\partial}^{X}$-closed and so is the wedge product
$$
\Theta^{i}\cdot\underline{\Theta}^{j}:=
\underbrace{\Theta\wedge\cdots\wedge\Theta}_{\mathrm{i-folds}}\wedge
\underbrace{\underline{\Theta}
\wedge\cdots\wedge\underline{\Theta}}_{\mathrm{j-folds}}
$$
 for any $0\leq i, j\leq n$.
Since
$\Theta^{k}-\underline{\Theta}^{k}=(\Theta-\underline{\Theta})
\cdot(\sum_{i+j=k-1}\Theta^{i}\cdot\underline{\Theta}^{j})$,
using the Leibniz rule deduces
$$
\Theta^{k}-\underline{\Theta}^{k}=\partial^{X}\overline{\partial}^{X}
\biggl[\biggl(\sum^{n}_{i=1}\eta_{i-1}\biggr)
\cdot\biggl(\sum_{i+j=k-1}\Theta^{i}
\cdot\underline{\Theta}^{j}\biggr)\biggr],
$$
for each $k\geq1$.
As a result, we are led to the conclusion that $c(\mathscr{E}, h)-c(\mathscr{E}, \underline{h})$ is $\partial^{X}\overline{\partial}^{X}$-exact.
\end{proof}
According to Proposition \ref{c-class} and (iv) in Theorem \ref{chern-cat}, the form $c(\mathscr{E}, h)$ represents a cohomology class in the real Bott--Chern cohomology $H^{(=)}_{BC}(X, \mathbb{R})$, denoted by $\{c(\mathscr{E})\}$,  which is independent with the choices of the metric $h$ and the non-canonical splitting \eqref{no-ca-spl}.
We call $\{c(\SE)\}$ the (total) Chern characteristic class of the antiholomorphic superconnection $\SE$ in the Bott--Chern cohomology.

Let $\CF^{\bullet}$ be an object in $\mathrm{D}^{\mathrm{b}}_{\mathrm{coh}}(X)$.
A result of Block \cite[Lemma 4.6]{Blo10} (see also \cite[Theorem 6.3.6]{BSW23}) says that there exists an antiholomorphic superconnection $\mathscr{E}=(E, A^{E^{\prime\prime}})$ such that $\mathscr{E}$ isomorphic to $\CF^{\bullet}$ in the derived category $\mathrm{D}^{\mathrm{b}}_{\mathrm{coh}}(X)$.
The following proposition is a direct consequence of  \cite[Theorem 8.8.1]{BSW23}.
\begin{prop}
The cohomology class $\{c(\mathscr{E})\}$ depends only on the isomorphism class of  $\CF^{\bullet}$ in $\mathrm{D}^{\mathrm{b}}_{\mathrm{coh}}(X)$.
\end{prop}
\begin{proof}
Assume that $\underline{\CF}^{\bullet}$ is isomorphic to $\CF^{\bullet}$ in the derived category $\mathrm{D}^{\mathrm{b}}_{\mathrm{coh}}(X)$.
Let $\mathscr{\underline{E}}=(\underline{E}, A^{\underline{E}^{\prime\prime}})$ be the corresponding antiholomorphic superconnection of $\underline{\CF}^{\bullet}$.
On account of \cite[Theorem 8.8.1]{BSW23}, the Chern characters $\mathrm{ch_{BC}}(A^{E^{\prime\prime}})$ and $\mathrm{ch_{BC}}(A^{\underline{E}^{\prime\prime}})$ are identical in $H^{(=)}_{BC}(X,\mathbb{R})$.
Consider the Chern character forms
$
\mathrm{ch}(A^{E_{0},2}, h)=\sum^{n}_{i=0}\omega_{i}
$
and
$
\mathrm{ch}(\underline{A}^{E_{0},2}, \underline{h})=\sum^{n}_{i=0}\underline{\omega}_{i}.
$
Let $\{\omega_{i}\}$ (resp.$\{\underline{\omega}_{i}\}$) be the corresponding Bott--Chern cohomology class of $\omega_{i}$ (resp. $\underline{\omega}_{i}$).
To be more specific, we have
$$\sum^{n}_{i=1}\{\omega_{i}\}
=\sum^{n}_{i=1}\{\underline{\omega}_{i}\}$$
and hence $\{\Theta\}=\{\underline{\Theta}\}$.
The equality $\{c(\mathscr{E})\}=\{c(\mathscr{\underline{E}})\}$ results from the definition of total Chern characteristic classes for antiholomorphic superconnections.
\end{proof}

Based upon the proposition above, we have the notion of Chern characteristic class of $\CF^{\bullet}$ in the Bott--Chern cohomology as follows.
\begin{defn}\label{t-real-b-c}
The (total) \emph{Chern characteristic class} of an object $\CF^{\bullet}\in\mathrm{D}^{\mathrm{b}}_{\mathrm{coh}}(X)$ is defined to be the cohomology class
$$
\hat{c}(\CF^{\bullet}):=\{c(\mathscr{E})\}\in H^{(=)}_{BC}(X, \mathbb{R}),
$$
where $\mathscr{E}=(E, A^{E^{\prime\prime}})$ is an antiholomorphic superconnection in the isomorphism class of $\CF^{\bullet}$ in $\mathrm{D}^{\mathrm{b}}_{\mathrm{coh}}(X)$.
\end{defn}

Let $\mathrm{B_{dg}}(X)$ be the dg-category of antiholomorphic superconnections on $X$ and $\mathrm{B}(X)$ its 0-cycle category.
Set $\mathrm{\underline{B}}(X)$ the homotopy category of $\mathrm{B}(X)$.
A result of Block \cite[Theorem 4.3]{Blo10} shows that there exists a natural equivalence
$$
\underline{F}_{X}:\mathrm{\underline{B}}(X)
\longrightarrow\mathrm{D^{b}_{coh}}(X)
$$
as triangulated categories.
The following result is a slight generalization of \cite[Proposition 1.5]{BC65}.
\begin{prop}[Whitney formula]\label{whitney}
Let $\CE^{\bullet}\rightarrow
\CF^{\bullet}\rightarrow\CG^{\bullet}
\rightarrow\CE^{\bullet}[1]$
be an exact triangle in $\mathrm{D^{b}_{coh}}(X)$.
Then their Chern characteristic classes satisfy the formula:
$$
\hat{c}(\CE^{\bullet})\cdot \hat{c}(\CG^{\bullet})
=\hat{c}(\CF^{\bullet}),
$$
\end{prop}
\begin{proof}
Consider the exact triangle in $\mathrm{D^{b}_{coh}}(X)$:
\begin{equation}\label{s-e-s}
\CE^{\bullet}\rightarrow
\CF^{\bullet}\rightarrow
\CG^{\bullet}\rightarrow\CE^{\bullet}[1].
\end{equation}
From the equivalence functor $\underline{F}_{X}$,
there is an exact triangle in $\mathrm{\underline{B}}(X)$ corresponding to \eqref{s-e-s}:
\begin{equation*}
(E, A^{E^{\prime\prime}})\rightarrow
(F, A^{F^{\prime\prime}})\rightarrow(G, A^{G^{\prime\prime}})
\rightarrow(E, A^{E^{\prime\prime}})[1].
\end{equation*}
Moreover, the functor $\underline{F}_{X}$ induces an isomorphism between the Grothendieck groups $K(\mathrm{\underline{B}}(X))$ and
$K(\mathrm{D^{b}_{coh}}(X))$.
According to \cite[Theorem 8.7.1]{BSW23}, the Chern character
\begin{eqnarray*}
  \mathrm{ch_{BC}} :K(\mathrm{\underline{B}}(X))&\longrightarrow& H^{(=)}_{BC}(X, \mathbb{R})\\
   (E, A^{E^{\prime\prime}})&\longmapsto&
   \{\mathrm{ch}(A^{E_{0},2}, h)\}
\end{eqnarray*}
is a morphism of groups.
As a result, we have
$$
\mathrm{ch_{BC}}(F, A^{F\prime\prime})=
\mathrm{ch_{BC}}(E, A^{E\prime\prime})+
\mathrm{ch_{BC}}(G, A^{G\prime\prime})
$$
which is equal to
$$
\mathrm{ch_{BC}}(\CF^{\bullet})=
\mathrm{ch_{BC}}(\CE^{\bullet})+
\mathrm{ch_{BC}}(\CG^{\bullet}).
$$
Comparing the definitions of the total Chern form \eqref{c-form} and the Chern character form \eqref{c-t-z-form} deduces the formula
$$
\hat{c}(\CE^{\bullet})\cdot \hat{c}(\CG^{\bullet})=\hat{c}(\CF^{\bullet})
$$
and this completes the proof.
\end{proof}

The following proposition is a direct consequence of \cite[Theorem 8.8.3]{BSW23}.
\begin{prop}[Naturality]\label{p-b-k}
Let $X$ and $Y$ be compact complex manifolds and $f:X\rightarrow Y$ a holomorphic map.
For any $\CF^{\bullet}\in\mathrm{D^{b}_{coh}}(Y)$, we have $Lf^{\ast}\CF^{\bullet}\in\mathrm{D^{b}_{coh}}(X)$ and
$$
\hat{c}(Lf^{\ast}\CF^{\bullet})
=f^{\ast}\hat{c}(\CF^{\bullet}),
$$
where $Lf^{\ast}$ is the left derived functor of $f^{\ast}$.
\end{prop}

Note that there is a canonical equivalence between the category of holomorphic vector bundles over $X$ and the category of locally free $\CO_{X}$-modules of constant rank.
Let $F$ be a holomorphic vector bundle on $X$ and $\CF$ the locally free $\CO_{X}$-modules corresponding to $F$.
Let $g^{F}$ be a Hermitian metric on $F$ and $\nabla^{F}$ the Chern connection corresponding to $g^{F}$.
Denote by $R^{F}$ the curvature of $\nabla^{F}$ which is a $(1,1)$-form valued in $\mathrm{End}(F)$.
According to \cite{BC65}, the total Chern form
$$
c(F, h)=\det\biggl(I+\frac{\sqrt{-1}}{2\pi}R^{F}\biggr)
$$
determines the \emph{refined Chern class} (also called the \emph{Chern class in the Bott--Chern cohomology}) of the holomorphic vector bundle:
$$
\hat{c}(F)=\{c(F, g^{F})\}\in H^{(=)}_{BC}(X, \mathbb{R}).
$$
Similarly, we have the refined Chern character of $F$ valued in the real Bott--Chern cohomology:
$$
\mathrm{ch_{BC}}(F)=
\biggl\{\mathrm{Tr}
\biggl[\exp\biggl(\frac{\sqrt{-1}}{2\pi}R^{F}\biggr)\biggr]\biggr\}
\in H^{(=)}_{BC}(X, \mathbb{R}).
$$
\begin{prop}
The cohomology classes $\hat{c}(\CF)$ and $\hat{c}(F)$ are identical in $H^{(=)}_{BC}(X, \mathbb{R})$.
\end{prop}
\begin{proof}
Note that
$$
\det\biggl(I+\frac{\sqrt{-1}}{2\pi}R^{F}\biggr)=
\exp\biggl(\mathrm{Tr}
\biggl[\log\biggl(I+\frac{\sqrt{-1}}{2\pi}R^{F}\biggr)\biggr]\biggr).
$$
Comparing the power series expansion formulae \eqref{log} and \eqref{exp}, to finish the proof, it is sufficient to show that $\mathrm{ch_{BC}}(F)$ coincides with $\mathrm{ch_{BC}}(\CF)$.
As a direct consequence of \cite[Theorem 9.4.1]{BSW23}, we get $\hat{c}(\CF)=\hat{c}(F)$.
\end{proof}

\section{Proof of Theorem \ref{thm1}}\label{sec4}
In this section, we first establish Riemann--Roch without denominators for refined Chern characteristic classes, and then we give the proof of Theorem \ref{thm1}.
In algebraic geometry, Grothendieck conjectured and proved Riemann--Roch without denominators in characteristic zero and the general case was proved by Jouanolou \cite{Jou70}.
In \cite[Theorem 3.1]{AH62}, Atiyah--Hirzebruch proved a generalization of the Riemann--Roch without denominators for Chern classes of complex manifolds.

\subsection{Riemann--Roch without denominators }\label{s-sec-4-1}
Throughout of this subsection, we assume that $Y$ is a compact complex manifold of dimension $n$ with a closed complex submanifold $X$ of codimension $r$,
and $F$ is a holomorphic vector bundle of rank $k$ over $X$.
Denote by $\CF$ the sheaf of holomorphic sections of $F$ which is a locally free $\CO_{X}$-modules.
Set $\imath:X\hookrightarrow Y$ the natural inclusion.
Then the direct image $\imath_{\ast}\CF$ is a coherent sheaf on $Y$.

We begin with the construction of deformation to the normal bundle for the closed imbedding $\imath:X\hookrightarrow Y$, see \cite[\S 9.2]{BSW23} or \cite[Section 4]{BGS90}.
Consider the pair $(Y\times\mathbb{P}^{1}, X\times\{\infty\})$.
We denote by $\varpi:\mathrm{Bl}_{X\times\{\infty\}}(Y\times\mathbb{P}^{1})
\rightarrow Y\times\mathbb{P}^{1}$ the blow-up of $Y\times\mathbb{P}^{1}$ with the center $X\times\{\infty\}$.
Then the exceptional divisor $\widetilde{E}=\varpi^{-1}(X\times\{\infty\})$ is isomorphic to the projective bundle\footnote{In this paper, for an arbitrary holomorphic vector bundle $E$ over a complex manifold $M$, we denote by $\mathbb{P}(E)$ the projective bundle of lines in $E$.} $P:=\mathbb{P}(N_{X\times\{\infty\}/Y\times\mathbb{P}^{1}})$ and therefore we identify $\widetilde{E}$ with $P$.
Let $\pi_{X}$ and $\pi_{\infty}$ be the projections from $X\times\{\infty\}$ to $X$ and $\infty$, respectively.
There exists an isomorphism of vector bundles
$$
N_{X\times\{\infty\}/Y\times\mathbb{P}^{1}}\cong
\pi^{\ast}_{X}(N)\oplus
\pi^{\ast}_{\infty}(N_{\infty/\mathbb{P}^{1}}).
$$
Put $A=\pi^{\ast}_{X}(N)\otimes
\pi^{\ast}_{\infty}(N^{-1}_{\infty/\mathbb{P}^{1}})$.
Then $P\cong\mathbb{P}(A\oplus\mathbb{C})$ and hence $\widetilde{E}$ can be thought of as the projective completion of the vector bundle
$A$ with the divisor $\mathbb{P}(N)$ at $\infty$.
Set $W=\mathrm{Bl}_{X\times\{\infty\}}(Y\times\mathbb{P}^{1})$.
We have the following blow-up diagram:
\begin{equation}\label{b-u-d-xp}
\vcenter{
\xymatrix@C=1.5cm{
P \ar[d]_{\varrho} \ar@{^{(}->}[r]^{\hat{\jmath}} & W\ar[d]^{\varpi}\\
 X \times\{\infty\}\ar@{^{(}->}[r]^{\hat{\imath}} & Y\times\mathbb{P}^{1}}
 }
\end{equation}
Here $\varrho$ is the restriction of $\varpi$ to $P$.
Combining \eqref{b-u-d-xp} with the blow-up diagram of the pair $(Y, X)$ derives a commutative cube:
\begin{equation}\label{bl-bl}
\vcenter{
\xymatrix@C=0.5cm{
  & \widetilde{E} \ar[rr] \ar'[d][dd]
      &  & W \ar[dd]        \\
  E \ar[ur]^{f}\ar[rr]^{\quad\jmath}\ar[dd]_{\rho}
      &  & \widetilde{Y} \ar[ur]^{g}\ar[dd] ^{\pi}\\
  & X\times\{\infty\} \ar'[r][rr]
      &  & Y\times\mathbb{P}^{1}                \\
  X \ar[rr]^{\imath}\ar[ur]^{\cong}
      &  & Y \ar[ur]^{h}       }
}
\end{equation}
Here $f$, $g$ and $h$ are natural embeddings.
Let $p_{1}$ and $p_{2}$ be the projections from $Y\times\mathbb{P}^{1}$ onto $Y$ and $\mathbb{P}^{1}$, respectively.
Set $q_{1}=
p_{1}\circ\varpi:W\rightarrow Y$ and
$q_{2}=p_{2}\circ\varpi:W\rightarrow\mathbb{P}^{1}$.
Then the fibres of $q_{2}$ are
$$
W_{t}:=q^{-1}_{2}(t)
=
\begin{cases}
P\cup\widetilde{Y}     &  t=\infty; \\
Y,      & t\neq\infty.
\end{cases}
$$
In particular, the intersection of $P$ and $\widetilde{Y}$ equals the exceptional divisor $E\cong\mathbb{P}(N)$.
Observe that $A$ is a holomorphic vector bundle over $X\times\{\infty\}$,
combining the canonical open embedding of $A$ in $P=\mathbb{P}(A\oplus 1)$ with the zero section embedding of $A$ derives a natural map $\overline{\imath}:X\times\{\infty\}\rightarrow A\subset P$.
Put $j_{0}:W_{0}=X\hookrightarrow W$ the inclusion map.
Because of the natural isomorphism  $X\times\{\infty\}\cong X$, we may identify $X\times\{\infty\}$ with $X$.
So that $F$ can be considered as a holomorphic vector bundle over $X\times\{\infty\}$ and $\CF$ is a coherent $\CO_{X\times\{\infty\}}$-modules.
Since $X\times\{\infty\}$ is a hypersurface in $X\times\mathbb{P}^{1}$, the blow-up of $X\times\mathbb{P}^{1}$ along $X\times\{\infty\}$ is identical with $X\times\mathbb{P}^{1}$, which can be considered as an embedded complex submanifold in $W$.
As a result, we have a natural embedding
$l:X\times\mathbb{P}^{1}=\mathrm{Bl}_{X\times\{\infty\}}(X\times\mathbb{P}^{1})\hookrightarrow W$.
The following commutative diagram (Figure \ref{def-nor-diagram}) will be used in tracing the argument in the proof.
\begin{figure}
\centering
\begin{tikzpicture}[scale=0.3]
\draw  (3,21)node[right] {$\mathbb{P}^{1}$};
\draw  (9,18)node[right] {$X\times\mathbb{P}^{1}$};
\draw  (9,11)node[right] {$Y=W_{0}$};
\draw  (18,11)node[right] {$X$};
\draw  (18,18)node[right] {$W$};
\draw  (17,14.6)node[right] {$X\times\mathbb{P}^{1}$};
\draw  (18,4)node[right] {$X$};
\draw  (24.8,4)node[right] {$X\times\mathbb{P}^{1}$};
\draw  (24.8,11)node[right] {$X\times\{\infty\}$};
\draw  (26,18)node[right] {$P$};
\draw  (6,20)node[right] {$p_{2}$};
\draw  (9.5,14.5)node[right] {$p_{1}$};
\draw  (15,11.5)node[right] {$\imath$};
\draw  (16.6,7)node[right] {$\mathrm{id}_{X}$};
\draw  (22,4.7)node[right] {$i_{X}$};
\draw  (22,8)node[right] {$p_{X}$};
\draw  (27,7)node[right] {$i_{\infty}$};
\draw  (22.5,11.5)node[right] {$=$};
\draw  (15,18.5)node[right] {$\varpi$};
\draw  (12,21)node[right] {$q_{2}$};
\draw  (18.8,16)node[right] {$l$};
\draw  (18.8,13)node[right] {$p_{X}$};
\draw  (22.5,18.7)node[right] {$\hat{\jmath}$};
\draw  (25.5,14.6)node[right] {$\varrho$};
\draw  (29.5,14.6)node[right] {$\bar{\imath}$};
\draw  (13,16)node[right] {$q_{1}$};
\draw  (15,14)node[right] {$j_{0}$};
\draw[-latex](9.5,18.2)--(4.3,20.6);
\draw[-latex](11.2,17)--(11.2,12);
\draw[-latex](18,18)--(13.5,18) ;
\draw[-latex] (26,18)-- (20,18);
\draw [-latex](19,18.6)to[out=157, in=0] (4.5,21);
\draw[-latex] (18,11)--(13.5,11);
\draw[-latex] (19,14.9)--(19,17);
\draw[-latex] (19,14)--(19,12);
\draw[-latex] (25,11)--(19.8,11);
\draw[-latex] (27,17.5)--(27,12);
\draw [-latex](29,12)to[out=60, in=0] (27.5,18);
\draw[-latex] (19,4.8)--(19,10);
\draw[-latex] (19.5,4)--(24.5,4);
\draw[-latex] (27,10)--(27,5);
\draw[-latex] (25.8,4.8)--(20,10);
\draw [-latex](18.5,17.5)to[out=-170, in=65] (11.7,12);
\draw [-latex](13,11.7)to[out=-710, in=250] (18.5,17.2);
\end{tikzpicture}
\caption{}
\label{def-nor-diagram}
\end{figure}

Let $S$ be the universal subbundle of the pullback $\varrho^{\ast}(A\oplus\mathbb{C})$ and $Q$ the universal quotient bundle.
Then we have an exact sequence of vector bundles over $P$:
\begin{equation*}
\xymatrix@C=0.5cm{
  0 \ar[r] & S \ar[r]^{} & \varrho^{\ast}(A\oplus\mathbb{C}) \ar[r]^{} & Q \ar[r] & 0. }
\end{equation*}
Let $\sigma:P\rightarrow Q$ be the canonical section determined by the projection of the trivial factor $1\in\mathbb{C}$ to $Q$.
It is noteworthy that $\sigma$ is a holomorphic section satisfying $\sigma^{-1}(0)=X\times\{\infty\}$.
Particularly, the Koszul complex $K(\sigma)=(\wedge^{\bullet}Q^{\vee})$ determined by $\sigma$ gives rise to a locally free resolution of the sheaf $\overline{\imath}_{\ast}\CO_{X\times\{\infty\}}$, where $Q^{\vee}$ is the dual of $Q$.
Since $\varrho^{\ast}F$ is a locally free sheaf on $P$, the functor $-\otimes\varrho^{\ast}F$ is exact and hence $(\wedge^{\bullet}Q^{\vee}\otimes\varrho^{\ast}F)$ yields a projection resolution of the sheaf $\overline{\imath}_{\ast}\CO_{X\times\{\infty\}}\otimes\varrho^{\ast}F$.
By the projection formula, we have
\begin{eqnarray*}
 \overline{\imath}_{\ast}\CO_{X\times\{\infty\}}\otimes\varrho^{\ast}F
 &\cong&  \overline{\imath}_{\ast}(\CO_{X\times\{\infty\}}\otimes
 \overline{\imath}^{\ast}\varrho^{\ast}F)\\
 &\cong& \overline{\imath}_{\ast}(\CO_{X\times\{\infty\}}\otimes (\varrho\circ\overline{\imath})^{\ast}F)\\
 &=& \overline{\imath}_{\ast}\CF,
\end{eqnarray*}
where the last equality results from the fact $\varrho\circ\overline{\imath}=\mathrm{id}_{X\times\{\infty\}}$.
Consequently, we obtain an explicit locally free resolution of $\overline{\imath}_{\ast}\CF$:
\begin{equation}\label{k-r}
\xymatrix@C=0.5cm{
  0 \ar[r] & \wedge^{r}Q^{\vee}\otimes\varrho^{\ast}F \ar[r]^{} & \cdots \ar[r]^{} & Q^{\vee}\otimes\varrho^{\ast}F \ar[r]^{} & \varrho^{\ast}F \ar[r]^{} & \overline{\imath}_{\ast}\CF \ar[r] & 0, }
\end{equation}
where $r=\mathrm{rank}(Q)$.
It follows from the definition of Grothendieck group that the following identity holds in $\mathrm{K}(\DC(P))$:
\begin{equation*}
[\overline{\imath}_{\ast}\CF]=
\sum_{i=0}^{r}(-1)^{i}[\wedge^{i}Q^{\vee}\otimes\varrho^{\ast}F].
\end{equation*}
According to Proposition \ref{whitney}, we get
\begin{equation*}
\BC(\overline{\imath}_{\ast}\CF)=
\prod_{i=0}^{r}\BC(\wedge^{i}Q^{\vee}
\otimes\varrho^{\ast}F)^{(-1)^{i}}.
\end{equation*}

Before continuing, we need a general lemma concerning the refined Chern characteristic classes of tensor products.
For an arbitrary complex manifold $X$, let $U$ and $V$ be two holomorphic vector bundles over $X$ with the rank $u$ and $v$, respectively.
Consider the tensor produce of $\wedge^{k}U^{\vee}\otimes V$, for any $0\leq k\leq u$.\

\begin{lem}\label{RR-domi-0}
There exists a unique power series $f(z_{1},\cdots,z_{u}; w_{1},\cdots,w_{v})$ with integer coefficients such that the following identity is valid in $H^{(=)}_{BC}(X, \mathbb{R})$:
\begin{equation*}\label{bar-imath-f-1}
\prod_{i=0}^{u}\BC\bigl(\wedge^{i}U^{\vee}
\otimes V\bigr)^{(-1)^{i}}-1=
\BC_{u}(U)\cdot f(U, V),
\end{equation*}
where $f(U,V)=f\bigl(\BC_{1}(U),\cdots,\BC_{u}(U);
\BC_{1}(V),\cdots,\BC_{v}(V)\bigr)$.
\end{lem}

\begin{proof}
The proof of \cite[Lemma 15.3]{Ful98} uses the formal factorization of Chern polynomial into Chern roots and this proof also holds for the refined Chern characteristic classes.
\end{proof}

We are ready to state the Riemann--Roch without denominators for refined Chern characteristic classes.
\begin{thm}\label{RR-domi}
The following identity is valid in $H^{(=)}_{BC}(Y, \mathbb{R})$:
$$
\hat{c}(\imath_{\ast}\CF)
=1+\imath_{\ast}(f(N,F)),
$$
where $N$ is the normal bundle of $X$ in $Y$ and $f(z_{1},\cdots,z_{r};w_{1},\cdots,w_{k})$ is the unique power series with integer coefficients determined by Lemma \ref{RR-domi-0}.
\end{thm}

The basic idea of the proof is a variation on the argument of \cite[Theorem 15.3]{Ful98}: use the locally free resolution \eqref{k-r} to prove the assertion for the embedding $\overline{\imath}$, and then apply the deformation to the normal bundle to establish the relationship between $\BC(\overline{\imath}_{\ast}\CF)\in H^{(=)}_{BC}(P, \mathbb{R})$ and $\BC(\imath_{\ast}\CF)\in H^{(=)}_{BC}(Y, \mathbb{R})$.

The following lemma will play an important role in the proof of Theorem \ref{RR-domi}.
\begin{lem}\label{key-lem}
With the notations above, we have
\begin{equation*}
\BC(\imath_{\ast}\CF)=\mu_{\ast}\BC(\overline{\imath}_{\ast}\CF),
\end{equation*}
where $\mu:P\rightarrow Y$ is the composition of $\hat{\jmath}$ and $q_{1}$ in the diagram \eqref{def-nor-diagram}.
\end{lem}
\begin{proof}
We will prove the assertion following the steps in the proof of  \cite[Theorem 9.3.1]{BSW23}.
Notice that the locally free sheaf $\CF$ is the sheaf of holomorphic sections of $F$ over $X$.
By the commutativity of \eqref{def-nor-diagram}, we have
\begin{equation*}
i^{\ast}_{X}\circ p^{\ast}_{X}\CF=(p_{X}\circ i_{X})^{\ast}\CF=\mathrm{id}^{\ast}_{X}\CF=\CF
\end{equation*}
and
\begin{equation*}
i^{\ast}_{\infty}\circ p^{\ast}_{X}\CF=(p_{X}\circ i_{\infty})^{\ast}\CF\cong\CF
\end{equation*}
in the derived category $\DC(X)$.
Consider the natural embeddings $\overline{\imath}:X\times\{\infty\}\rightarrow W$,  $j_{0}:Y=W_{0}\rightarrow W$, and $\hat{\jmath}:P\rightarrow W$.
There exists a commutative diagram:
\begin{equation*}\label{trans-1}
\vcenter{
\xymatrix@=1.0cm{
  (X\times\mathbb{P}^{1})\cap Y=X \ar[d]_{\imath} \ar[r]^{\quad\quad i_{X}} & X\times\mathbb{P}^{1} \ar[d]_{l} & X\times\{\infty\}= (X\times\mathbb{P}^{1})\cap P \ar[l]_{i_{\infty}\quad\quad\quad}\ar[d]^{\overline{\imath}} \\
  Y=W_{0}\ar[r]^{j_{0}} & W  & P\ar[l]_{\hat{\jmath}}  }
}
\end{equation*}
In particular, $X\times\mathbb{P}^{1}$ and $W_{0}$ are transverse which means $N_{X/Y}=N_{X\times\mathbb{P}^{1}/W}|_{X}$.
Owing to a result \cite[Proposition 9.1.1]{BSW23} by Bismut--Shen--Wei, the following equation holds in $\DC(Y)$:
\begin{equation}\label{imath-f}
Lj^{\ast}_{0}\circ l_{\ast}(p^{\ast}_{X}\CF)
\cong\imath_{\ast}\circ Li^{\ast}_{X}(p^{\ast}_{X}\CF)
=\imath_{\ast}(p_{X}\circ i_{X})^{\ast}\CF\cong\imath_{\ast}\CF.
\end{equation}
Likewise, since $X\times\mathbb{P}^{1}$ and $P$ are transverse we obtain the following equation in $\DC(P)$:
\begin{equation}\label{barim-f}
L\hat{\jmath}^{\ast}\circ l_{\ast}(p^{\ast}_{X}\CF)
\cong\overline{\imath}_{\ast}\circ Li^{\ast}_{\infty}(p^{\ast}_{X}\CF)
=\overline{\imath}_{\ast}(p_{X}\circ i_{\infty})^{\ast}\CF\cong\overline{\imath}_{\ast}\CF.
\end{equation}
Consider the Chern characteristic classes of \eqref{imath-f} and \eqref{barim-f}.
Using Proposition \ref{p-b-k} deduces
\begin{equation}\label{c-hat-f}
j^{\ast}_{0}\BC( l_{\ast}(p^{\ast}_{X}\CF))=\BC(\imath_{\ast}\CF)
\,\,\,\,\mathrm{and}\,\,\,\,
\hat{\jmath}^{\ast}\BC( l_{\ast}(p^{\ast}_{X}\CF))=\BC(\overline{\imath}_{\ast}\CF).
\end{equation}

Let $\alpha\in\Omega^{(=)}(W,\mathbb{R})$ be a $d^{W}$-closed form representing the Chern characteristic class $\BC( l_{\ast}(p^{\ast}_{X}\CF))$.
Note that $q_{2}:W\rightarrow\mathbb{P}^{1}$ is a submersion and $W_{0}=q^{-1}_{2}(0)=Y$ and $W_{\infty}=q^{-1}_{2}(\infty)=P\cup\widetilde{Y}$.
Denote by $[W_{0}]=[Y]$ and $[W_{\infty}]=[P]+[\widetilde{Y}]$ the currents of integration on $W_{0}$ and $W_{\infty}$, respectively.
Let $x\in\mathbb{C}$ be the canonical meromorphic coordinate on $\mathbb{P}^{1}$ vanishing at $0$.
Then we have the classical Poincar\'{e}--Lelong equation
$$
\frac{\overline{\partial}^{\mathbb{P}^{1}}\partial^{\mathbb{P}^{1}}}{2\pi\sqrt{-1}}\log\bigl(|z|^{2}\bigr)=[0]-[\infty].
$$
As a result, we get
\begin{eqnarray*}
 \frac{\overline{\partial}^{W}\partial^{W}}{2\pi\sqrt{-1}}\biggl[q^{\ast}_{2}\log\bigl(|z|^{2}\bigr)\biggr]
  &=&q^{\ast}_{2}[0]-q^{\ast}_{2}[\infty] = [q^{-1}_{2}(0)]-[q^{-1}_{2}(\infty)] \\
  &=& [Y]-[P]-[\widetilde{Y}].
\end{eqnarray*}
Furthermore, since $\alpha$ is $d^{W}$-closed we have the following equation of currents on $W$:
\begin{eqnarray}\label{p-l-equ}
 \frac{\overline{\partial}^{W}\partial^{W}}{2\pi\sqrt{-1}}\biggl[\alpha\cdot q^{\ast}_{2}\log\bigl(|z|^{2}\bigr)\biggr]
 &=&\alpha\cdot[Y]-\alpha\cdot[P]-\alpha\cdot[\widetilde{Y}]\nonumber\\
&=& j_{0,\ast}(j^{\ast}_{0}\alpha)-\hat{\jmath}_{\ast}(\hat{\jmath}^{\ast}\alpha)-g_{\ast}(g^{\ast}\alpha),
\end{eqnarray}
where $g:\widetilde{Y}\rightarrow W$ is the natural embedding in \eqref{bl-bl}.
The second equality results from the projection formula.
Taking $q_{1,\ast}$ to \eqref{p-l-equ} gives rise to an equation of currents on $Y$ (cf. \cite[(9.3.11)]{BSW23}):
\begin{eqnarray}\label{p-l-equ-Y}
 \frac{\overline{\partial}^{Y}\partial^{Y}}{2\pi\sqrt{-1}}\biggl[q_{1,\ast}\bigl(\alpha\cdot q^{\ast}_{2}\log\bigl(|z|^{2}\bigr)\bigr)\biggr]
&=&(q_{1}\circ j_{0})_{\ast}(j^{\ast}_{0}\alpha)-(q_{1}\circ\hat{\jmath})_{\ast}(\hat{\jmath}^{\ast}\alpha)-(q_{1}\circ g)_{\ast}(g^{\ast}\alpha)\nonumber\\
&=&j^{\ast}_{0}\alpha-\mu_{\ast}(\hat{\jmath}^{\ast}\alpha)-(q_{1}\circ g)_{\ast}(g^{\ast}\alpha).
\end{eqnarray}
Here the last equality comes from the facts $q_{1}\circ j_{0}=\mathrm{id}_{Y}$ and $\mu=q_{1}\circ\hat{\jmath}$.
On the one hand, we have $\BC(Lg^{\ast}\circ l_{\ast}(p^{\ast}_{X}\CF))=g^{\ast}\BC( l_{\ast}(p^{\ast}_{X}\CF))=\{g^{\ast}\alpha\}$.
On the other hand, since $p^{\ast}_{X}\CF$ is supported on $X\times\mathbb{P}^{1}$ and $\widetilde{Y}\cap (X\times\mathbb{P}^{1})=\emptyset$ we get $Lg^{\ast}\circ l_{\ast}(p^{\ast}_{X}\CF)\cong0$ in $\DC(\widetilde{Y})$.
This implies $\{g^{\ast}\alpha\}=0$ in $H^{(=)}_{BC}(\widetilde{Y}, \mathbb{R})$ and thus $\{(q_{1}\circ g)_{\ast}(g^{\ast}\alpha)\}=0$ in $H^{(=)}_{BC}(Y, \mathbb{R})$.
From \eqref{c-hat-f} and \eqref{p-l-equ-Y}, we obtain the identity $\BC(\imath_{\ast}\CF)=\mu_{\ast}\BC(\overline{\imath}_{\ast}\CF)$ and this completes the proof.
\end{proof}

We are now in a position to prove Theorem \ref{RR-domi}.

\begin{proof}[Proof of Theorem \ref{RR-domi}]
Observe that $\sigma:P\rightarrow Q$ is a holomorphic section such that the zero locus $\sigma^{-1}(0)=X\times\{\infty\}$ having the expected dimension $n-r$.
Denote by $[X\times\{\infty\}]\in\mathfrak{D}^{r,r}(P, \mathbb{R})$ the current of integration over $X\times\{\infty\}$ which represents a Bott--Chern cohomology class in $H_{BC}^{r,r}(\mathfrak{D}^{(=)}(P, \mathbb{R}))$.
Particularly, we have $\overline{\imath}_{\ast}1=[X\times\{\infty\}]_{\mathrm{BC}}$, where $1\in H^{0,0}_{BC}(X\times\{\infty\}, \mathbb{R})$.
Given a Hermitian metric $g^{Q}$ and let $\nabla^{Q}$ be the Chern connection of $g^{Q}$.
The top Chern form $c_{r}(\nabla^{Q})$ is a $d^{P}$-closed $(r,r)$-form on $P$.
On account of the generalized Poincar\'{e}--Lelong formula for holomorphic vector bundles (cf. \cite[Theorem 1.1]{An07} or \cite[Theorem 1.3]{CGL20}), there exists a current $T\in\mathfrak{D}^{r-1,r-1}(P, \mathbb{R})$ such that
$$
c_{r}(\nabla^{Q})-[X\times\{\infty\}]=\partial^{P}\overline{\partial}^{P}T.
$$
This implies that the top refined Chern characteristic class $\BC_{r}(Q)$ and the Bott--Chern cohomology class $[X\times\{\infty\}]_{\mathrm{BC}}$ are identical under the smoothing of Bott--Chern cohomology
$$
H^{\ast,\ast}_{BC}(P, \mathbb{R})\cong
H^{\ast,\ast}_{BC}(\mathfrak{D}^{(=)}(P, \mathbb{R})).
$$
Due to Lemma \ref{RR-domi-0}, there exists a unique power series $f(z_{1},\cdots,z_{r}; w_{1},\cdots,w_{k})$ with integer coefficients such that
\begin{equation}\label{b-i-f}
\BC(\overline{\imath}_{\ast}\CF)-1=
\prod_{i=0}^{r}\BC\bigl(\wedge^{i}Q^{\vee}
\otimes\varrho^{\ast}F\bigr)^{(-1)^{i}}-1=
\BC_{r}(Q)\cdot f(Q, \varrho^{\ast}F).
\end{equation}
Using the projection formula (cf. \cite[Theorem 2.14]{Dem12}), we have
\begin{eqnarray}\label{bar-imath-f-2}
\BC_{r}(Q)\cdot f(Q, \varrho^{\ast}F)
  &=& [X\times\{\infty\}]_{\mathrm{BC}}\cdot f(Q, \varrho^{\ast}F) \nonumber\\
  &=& \overline{\imath}_{\ast}1\cdot f(Q, \varrho^{\ast}F) \nonumber\\
  &=& \overline{\imath}_{\ast}\bigl(1\cdot\overline{\imath}^{\ast} f(Q, \varrho^{\ast}F)\bigr)\\
  &=&\overline{\imath}_{\ast}f(\overline{\imath}^{\ast}Q, \overline{\imath}^{\ast}\varrho^{\ast}F)\nonumber\\
  &=&\overline{\imath}_{\ast}f(N, F).\nonumber
\end{eqnarray}
Here the last equality results from the facts $\overline{\imath}^{\ast}Q\cong N$ and $\overline{\imath}^{\ast}\varrho^{\ast}F=F$.
From \eqref{b-i-f} and \eqref{bar-imath-f-2}, we get
\begin{equation}\label{bar-imath-f-3}
\BC(\overline{\imath}_{\ast}\CF)=1+\overline{\imath}_{\ast}f(N, F).
\end{equation}
Notice that $\mu\circ\overline{\imath}=q_{1}\circ\hat{\jmath}\circ\overline{\imath}=\imath$.
Combining Lemma \ref{key-lem} with \eqref{bar-imath-f-3} yields to the conclusion
\begin{equation*}
\BC(\imath_{\ast}\CF)=1+\mu_{\ast}\circ\overline{\imath}_{\ast}(f(N,F))
=1+\imath_{\ast}(f(N,F))
\end{equation*}
and this completes the proof.
\end{proof}

The following example is an analogy of \cite[Example 15.3.1]{Ful98}.
\begin{ex}
Observe that the Bott--Chern cohomology class $f(N,F)$ can be written as
$$
f(N,F)=\sum_{q\geq0}f_{q}(N,F),
$$
where $f_{q}(N,F)\in H^{q,q}_{BC}(X, \mathbb{R})$.
So we get
$
\BC_{q}(\imath_{\ast}\CF)=\imath_{\ast}(f_{q-r}(N, F))
$
and hence the vanishing result $\BC_{q}(\imath_{\ast}\CF)=0$ for any $1\leq q\leq r-1$.
From the definition of $f(N,F)$, we obtain
$$
f_{0}(N,F)=(-1)^{r-1}(r-1)!\,k.
$$
This implies
$$
\BC_{r}(\imath_{\ast}\CF)=(-1)^{r-1}(r-1)!\,k\cdot\imath_{\ast}(1)=(-1)^{r-1}(r-1)!k\cdot[X]_{\mathrm{BC}},
$$
where $1\in H^{0,0}_{BC}(X,\mathbb{R})$.
Particularly, if $\CF=\CO_{X}$ then we have
$$
\BC_{r}(\imath_{\ast}\CO_{X})=(-1)^{r-1}(r-1)!\cdot[X]_{\mathrm{BC}}.
$$
\end{ex}

\subsection{Proof of Theorem \ref{thm1}}
The basic idea of the proof is the same as \cite[Theorem 2]{Por60}, see also \cite[Theorem 15.4]{Ful98}.
Let $Y$ be a compact complex manifold with dimension $n$ and $\imath:X\hookrightarrow Y$ a closed complex submanifold with codimension $r\geq2$.
Set $N=N_{X/Y}$.
Recall the construction of the blow-up of $Y$ along $X$.
There exists a commutative diagram as follows.
\begin{equation*}
\vcenter{
\xymatrix@=1.0cm{
E \ar[d]_{\rho} \ar@{^{(}->}[r]^{\jmath} & \widetilde{Y}\ar[d]^{\pi}\\
 X \ar@{^{(}->}[r]^{\imath} & Y}
 }
\end{equation*}
Here $E\cong\mathbb{P}(N)$, $\jmath: E \hookrightarrow \widetilde{Y}$ is the inclusion and $\rho=\pi|_{E}$.
From definition, we have $N_{E/\widetilde{Y}}\cong\CO_{E}(-1)$ and the tautological exact sequence $E$:
\begin{equation}\label{tau-ex}
\xymatrix@C=0.5cm{
  0 \ar[r] & \CO_{E}(-1) \ar[r]^{} & \rho^{\ast}N \ar[r]^{} & Q_{E} \ar[r] & 0,}
\end{equation}
where $Q_{E}$ is the universal quotient bundle on $E$.
Tensoring $\CO_{E}(1)$ to \eqref{tau-ex} deduces
\begin{equation}\label{tau-ex-1}
\xymatrix@C=0.5cm{
  0 \ar[r] & \CO_{E} \ar[r]^{} & \rho^{\ast}N\otimes\CO_{E}(1) \ar[r]^{} & Q_{E}\otimes\CO_{E}(1) \ar[r] & 0.}
\end{equation}

To finish the proof, we need the following lemma, see \cite[Lemma 15.4]{Ful98} for non-singular algebraic varieties.
\begin{lem}\label{4-ext-seq}
There are exact sequences:
\begin{itemize}
\item[(i)] $\xymatrix@C=0.5cm{
  0 \ar[r] & \CO_{E} \ar[r]^{} & \rho^{\ast}N\otimes\CO_{E}(1) \ar[r]^{} & TE \ar[r]^{} &\rho^{\ast}TX \ar[r] & 0 };$
\item[(ii)] $\xymatrix@C=0.5cm{
  0 \ar[r] & TE \ar[r]^{} &\jmath^{\ast}T\widetilde{Y} \ar[r]^{} &\CO_{E}(-1) \ar[r] & 0 };$
\item[(iii)] $\xymatrix@C=0.5cm{
  0 \ar[r] & T\widetilde{Y} \ar[r]^{} &\pi^{\ast}TY \ar[r]^{} &\jmath_{\ast}\mathcal{Q}_{E} \ar[r] & 0 }.$
\end{itemize}
Here $\mathcal{Q}_{E}$ is the sheaf of holomorphic sections of $Q_{E}$ and the third is an exact sequence of sheaves on $\widetilde{Y}$.
\end{lem}

Consider the natural embedding $\jmath:E\rightarrow\widetilde{Y}$.
On the one hand, applying the Whitney formula (see Proposition \ref{whitney}) to the third exact sequence in Lemma \ref{4-ext-seq} implies
\begin{equation}\label{426}
\pi^{\ast}\BC(Y)=\BC(\widetilde{Y})\cdot\BC(\jmath_{\ast}\mathcal{Q}_{E}).
\end{equation}
On the other hand, due to Theorem \ref{RR-domi}, we obtain
\begin{equation}\label{427}
\BC(\jmath_{\ast}\mathcal{Q}_{E})=1+\jmath_{\ast}(f(\CO_{E}(-1), Q_{E}))
\end{equation}
since $N_{E/\widetilde{Y}}\cong\CO_{E}(-1)$.
Combining \eqref{426} with \eqref{427} deduces
\begin{eqnarray*}
\pi^{\ast}\BC(Y)-\BC(\widetilde{Y})
  &=&\BC(\widetilde{Y})\cdot\jmath_{\ast}(f(\CO_{E}(-1), Q_{E}))= \jmath_{\ast}\bigl[\jmath^{\ast}\BC(\widetilde{Y})\cdot f(\CO_{E}(-1), Q_{E})\bigr].
\end{eqnarray*}
From Lemma \ref{RR-domi-0}, in which $U$ and $V$ are replaced by $\CO_{E}(-1)$ and $Q_{E}$ respectively, we have
\begin{equation}\label{428}
 \frac{\BC(Q_{E})}{\BC(\CO_{E}(1)\otimes Q_{E})}-1=\BC_{1}(\CO_{E}(-1))\cdot f(\CO_{E}(-1), Q_{E}).
\end{equation}
Applying the Whitney formula to the first and the second exact sequences in Lemma \ref{4-ext-seq} derives
$$
\BC(E)=\BC(\rho^{\ast}N\otimes\CO_{E}(1))\cdot\rho^{\ast}\BC(X)
$$
and
$$
\jmath^{\ast}\BC(\widetilde{Y})=\BC(E)\cdot\BC(\CO_{E}(-1))
$$
which implies
\begin{equation}\label{429}
\jmath^{\ast}\BC(\widetilde{Y})=\BC(\rho^{\ast}N\otimes\CO_{E}(1))\cdot\rho^{\ast}\BC(X)\cdot\BC(\CO_{E}(-1)).
\end{equation}
Using the Whitney formula again, from \eqref{tau-ex} and \eqref{tau-ex-1}, we get
\begin{equation}\label{430}
\BC(\rho^{\ast}N)=\BC(\CO_{E}(-1))\cdot\BC(Q_{E})
\end{equation}
and
\begin{equation}\label{431}
\BC(\rho^{\ast}N\otimes\CO_{E}(1))=\BC(Q_{E}\otimes\CO_{E}(1)).
\end{equation}
Put $\zeta=\BC_{1}(\CO_{E}(1))$.
By \eqref{428}-\eqref{431}, we have
$$
\jmath^{\ast}\BC(\widetilde{Y})\cdot f(\CO_{E}(-1), Q_{E})=\rho^{\ast}\BC(X)\cdot\alpha
$$
where
$$
\alpha=\frac{1}{\zeta}\cdot\biggl[\rho^{\ast}\BC(N)-(1-\zeta)\cdot\BC(\rho^{\ast}N\otimes\CO_{E}(1))\biggr].
$$
It remains to compute the refined Chern characteristic classes of the tensor product $\rho^{\ast}N\otimes\CO_{E}(1)$.
Based on the projective bundle formula for Bott--Chern cohomology (cf. \cite[Corollary 3.3]{YY23}), we have the splitting construction for refined Chern characteristic classes of holomorphic vector bundles, which is analogous to the Chern classes of vector bundles over algebraic schemes (cf. \cite[Page 51]{Ful98}).
Consequently, the splitting principle for refined Chern characteristic classes holds.
Assume that the Bott--Chern polynomial of $\rho^{\ast}N$ is factored as
$$
\BC(\rho^{\ast}N, t):=\sum^{r}_{k=0}\BC_{k}(\rho^{\ast}N)\cdot t^{k}=\prod^{r}_{k=1}(1+\alpha_{k}t),
$$
we call $\alpha_{0},\cdots,\alpha_{k}$ the refined Chern roots of $\rho^{\ast}N$.
Observe that the refined Chern root of $\CO_{E}(1)$ is  $\zeta$.
By definition, the refined Chern roots of the tensor product $\rho^{\ast}N\otimes\CO_{E}(1)$ are the sums $\alpha_{0}+\zeta,\cdots,\alpha_{k}+\zeta$.
Denote by $\mathfrak{s}_{k}(\alpha_{1},\cdots,\alpha_{r} )$ the $k$-th elementary symmetric function of $\alpha_{1},\cdots,\alpha_{r} $ which is equal to the refined Chern characteristic class $\BC_{k}(\rho^{\ast}N)$.
Due to the identity
$$
\prod^{r}_{k=1}\bigl(1+(\zeta+\alpha_{k})\cdot t\bigr)=\sum^{r}_{k=0}(1+\zeta\cdot t)^{r-k}\mathfrak{s}_{k}(\alpha_{1},\cdots,\alpha_{r} )\cdot t^{k},
$$
we get
$$
\BC(\rho^{\ast}N\otimes\CO_{E}(1))=\sum^{r}_{i=0}(1+\zeta)^{i}\rho^{\ast}\hat{c}_{r-i}(N).
$$
Consequently, we can rewrite $\alpha$ as
\begin{eqnarray*}
\alpha&=&\frac{1}{\zeta}\cdot\biggl[\sum^{r}_{k=0}\rho^{\ast}\hat{c}_{r-k}(N)
-(1-\zeta)\sum^{r}_{k=0}(1+\zeta)^{k}\rho^{\ast}\hat{c}_{r-k}(N)\biggr]\\
&=& \sum^{r}_{k=0}\biggl[\rho^{\ast}\BC_{r-k}(N)\cdot\biggl(\sum^{k}_{i=0}\binom{i}{k}\zeta^{i}-\sum^{k}_{i=1}\binom{i}{k}\zeta^{i-1}\biggr)\biggr].
\end{eqnarray*}
This finishs the proof of Theorem \ref{thm1}.
\begin{rem}
As a direct consequence of the RRG for closed embeddings \cite[Theorem 9.3.1]{BSW23}, we get the blow-up formula of Chern characters:
$$
\pi^{\ast}\ch(Y)
      -\ch(\widetilde{Y})
      =\jmath_{\ast}\biggl
      [\frac{1-\exp(-\zeta)}{\zeta}\bigl(\rho^{\ast}\ch(N)
      -\exp(\zeta)\bigr)
      \biggr].
$$
\end{rem}
\section{Rational case}\label{sec5}
Let $X$ be a compact complex manifold of complex dimension $n$.
For any integers $0\leq p,q\leq n$, the rational Bott--Chern complex $\mathcal{B}_{p,q,\mathbb{Q}}^{\bullet}$ is defined to be:
\begin{equation*}
\xymatrix@C=0.3cm{
0\ar[r]& \mathbb{Q}\ar[r]^{\Delta\quad\quad} & \CO_{X}\oplus \overline{\CO}_{X} \ar[r]^{} &\Omega_{X}^{1}\oplus \overline{\Omega_{X}^{1}}
  \ar[r]^{}& \cdots  \ar[r]^{}& \Omega_{X}^{p-1}\oplus \overline{\Omega_{X}^{p-1} } \ar[r]^{} & \overline{\Omega_{X}^{p}} \ar[r]^{} & \cdots  \overline{\Omega_{X}^{q-1}} \ar[r]^{} & 0,}
\end{equation*}
where  $\Delta$ is multiplication by 1 for the first component and multiplication by -1 for the second component (cf. \cite{Sch07}).
The $(p+q)$-th hypercohomology of the complex above $H^{p,q}_{BC}(X,\mathbb{Q}):=\mathbb{H}^{p+q}(X, \mathcal{B}_{p,q,\mathbb{Q}}^{\bullet})$ is called the {\it rational Bott--Chern cohomology} of bidegree $(p,q)$.
It is noteworthy that the rational Bott--Chern complex is quasi-isomorphic to a complex formed by locally integral currents and distributions.
The latter complex is soft which implies that the rational Bott--Chern cohomology can be represented by global sections of this complex (cf. \cite[Remark 1]{Wu23}).

Note that the set of all antiholomorphic superconnections naturally carries the structure of a complex affine space.
It is not immediately clear how to define the Chern characters in rational Bott--Chern cohomology in terms of curvature forms of superconnections.
Set
$$
H^{(=)}_{BC}(X,\mathbb{Q})=\bigoplus_{0\leq p\leq n}H^{p,p}_{BC}(X,\mathbb{Q}).
$$
Using the  axiomatic approach of Grivaux \cite{Gri10}, the first named author constructed a rational Chern character $\chr:K_{0}(X)\rightarrow H^{(=)}_{BC}(X, \mathbb{Q})$, where $K_{0}(X)$ is the Grothendieck group of the coherent sheaves on $X$.
As a result, for any coherent $\CO_{X}$-modules $\CF$ over $X$, we can define its Chern classes as elements in $H^{(=)}_{BC}(X,\mathbb{Q})$ as follows.

Given formal variables $\underline{x}:=(x_{1},\cdots,x_{r})$, for each $k\geq0$, recall that the Newton sum $S_{k}(\underline{x})$ is defined to be
$$
S_{k}(\underline{x})=\frac{1}{k!}(x^{k}_{1}+\cdots+x^{k}_{r}),
$$
and the symmetric sums $\sigma_{k}(\underline{x})$ is defined by setting
$$
\prod_{i=1}^{r}(t+x_{i})=\sum_{k=0}^{r}\sigma_{k}(\underline{x})t^{r-k}.
$$
Here $\sigma_{k}(\underline{x})=0$ when $k>r$.
In particular, for any $p\geq0$,  there exists $P_{m}$ and $Q_{m}$ in $\mathbb{Q}[T_{1},\cdots,T_{n}]$ characterized by the identities $S_{m}(\underline{x})=
P_{m}(\sigma_{1}(\underline{x}),\cdots,\sigma_{n}(\underline{x}))$
and
$\sigma_{m}(\underline{x})=
Q_{m}(S_{1}(\underline{x}),\cdots, S_{n}(\underline{x}))$.
Put
$$
\chr(\CF)=\sum_{p=0}^{n}\mathrm{ch}_{p}(\CF)_{\mathbb{Q}},
$$
where $\mathrm{ch}_{p}(\CF)_{\mathbb{Q}}\in H^{p,p}_{BC}(X, \mathbb{Q})$.
The \emph{rational Chern classes of $\CF$} are defined to be $\BC_{0}(\CF)_{\mathbb{Q}}=1$ and
$$
\BC_{p}(\CF)_{\mathbb{Q}}:
=Q_{p}(\mathrm{ch}_{1}(\CF)_{\mathbb{Q}},\cdots,
\mathrm{ch}_{n}(\CF)_{\mathbb{Q}}),
$$
when $1\leq p\leq n$.
Naturally, one may wonder if the blow-up formula for rational Chern classes still hold.
In the rest of this section, we will show that the arguments in the proof of Theorem \ref{thm1} is also valid for the rational case.

Consider $\DC(X)$ the derived category of bounded complexes of $\CO_{X}$-modules having coherent cohomology.
Denote by $K_{0}(\DC(X))$ the associated Grothendieck group.
There exists a natural morphism
$$
\DC(X)\ni\CF^{\bullet}\longmapsto \sum_{k}(-1)^{k}[\mathcal{H}^{k}(\CF^{\bullet})]\in K_{0}(X)
$$
which induces a canonical isomorphism
$K_{0}(\DC(X))\cong K_{0}(X)$ (cf. \cite[Lemma 13.28.2]{Sta}).
Combining with $\chr:K_{0}(X)\rightarrow H^{(=)}_{BC}(X, \mathbb{Q})$ derives the following definition.
\begin{defn}\label{t-r-b-c}
For any $\CF^{\bullet}\in\DC(X)$, the rational Chern character of $\CF^{\bullet}$ is defined to be:
$$
\chr(\CF^{\bullet})=\sum_{k}(-1)^{k}\chr(\mathcal{H}^{k}(\CF^{\bullet}))
\in H^{(=)}_{BC}(X, \mathbb{Q}).
$$
Moreover, the (total) rational Chern class of $\CF^{\bullet}$ is given by
$$
\BC(\CF^{\bullet})_{\mathbb{Q}}
=\prod_{k}\bigl(\BC(\mathcal{H}^{k}(\CF^{\bullet}))_{\mathbb{Q}}\bigr)^{(-1)^{k}}
\in H^{(=)}_{BC}(X, \mathbb{Q}).
$$
\end{defn}
Assume that $Y$ is another compact complex manifold and $f:X\rightarrow Y$ is a holomorphic map.
Then we get an induced homomorphism of Grothendieck groups:
\begin{eqnarray*}
  f^{!}:K_{0}(Y) &\longrightarrow & K_{0}(X)\\
  { [\CF]}&\longmapsto&
  \sum_{k}(-1)^{k}[\mathcal{H}^{k}(Lf^{\ast}\CF)].
\end{eqnarray*}
On the one hand, it follows from the definition of the functor $Lf^{\ast}$ and the morphism $f^{!}$ that there is a commutative diagram:
\begin{equation}\label{k-comm-0}
\vcenter{
\xymatrix{
   \DC(Y)\ar[d]_{} \ar[r]^{Lf^{\ast}} & \DC(X) \ar[d]^{} \\
  K_{0}(Y)  \ar[r]^{f^{!}} & K_{0}(X).   }
  }
\end{equation}
On the other hand, by the axiomic approach of Grivaux, the following diagram is commutative.
\begin{equation}\label{k-comm-1}
\vcenter{
\xymatrix{
  K_{0}(Y) \ar[d]_{\chr} \ar[r]^{f^{!}} & K_{0}(X) \ar[d]^{\chr} \\
  H^{(=)}_{BC}(Y, \mathbb{Q})  \ar[r]^{f^{\ast}} & H^{(=)}_{BC}(X, \mathbb{Q})  }
  }
\end{equation}
From \eqref{k-comm-0} and \eqref{k-comm-1}, for any $\CF^{\bullet}\in\DC(Y)$, the following identity is valid:
$$
\chr(Lf^{\ast}\CF^{\bullet})=f^{\ast}\chr(\CF^{\bullet}).
$$
As a direct consequence, the identity in Proposition \ref{p-b-k} holds in the rational case, i.e.,
$$
\hat{c}(Lf^{\ast}\CF^{\bullet})_{\mathbb{Q}}
=f^{\ast}\hat{c}(\CF^{\bullet})_{\mathbb{Q}}.
$$

Next we will show that the Whitney formula holds in the rational case.
For any short exact sequence of coherent $\CO_{X}$-modules
$\xymatrix@C=0.3cm{
  0 \ar[r] & \CE \ar[r]^{} & \CF\ar[r]^{} & \CG\ar[r] & 0 ,}$
according to \cite[Corollary 1]{Wu23} the identity
$\chr(\CF)=\chr(\CE)+\chr(\CG)$ holds, and therefore we get
$\BC(\CF)_{\mathbb{Q}}=\BC(\CE)_{\mathbb{Q}}\cdot\BC(\CG)_{\mathbb{Q}}$.
Consider an exact triangle in the category $\mathrm{D^{b}_{coh}}(X)$:
 $\CE^{\bullet}\rightarrow\CF^{\bullet}\rightarrow\CG^{\bullet}
\rightarrow\CE^{\bullet}[1].$
There is a long exact sequence of cohomology sheaves:
\begin{equation}\label{l-ex-coh}
\xymatrix@C=0.5cm{
   \cdots\ar[r] & \mathcal{H}^{k}(\CE^{\bullet}) \ar[r]^{} & \mathcal{H}^{k}(\CF^{\bullet}) \ar[r]^{} & \mathcal{H}^{k}(\CG^{\bullet}) \ar[r]^{} & \mathcal{H}^{k+1}(\CE^{\bullet}) \ar[r]^{} & \cdots. }
\end{equation}
From Definition \ref{t-r-b-c} and the exactness of \eqref{l-ex-coh}, we obtain the Whitney formula:
$$
\hat{c}(\CE^{\bullet})_{\mathbb{Q}}\cdot \hat{c}(\CG^{\bullet})_{\mathbb{Q}}=
\hat{c}(\CF^{\bullet})_{\mathbb{Q}}.
$$

In the proof of Riemann--Roch without denominators (Theorem \ref{RR-domi}), the key points are: the generalized Poincar\'{e}--Lelong formula which relates the top refined Chern characteristic classes to the cycle classes associated to the zero locus of some holomorphic sections and Lemma \ref{key-lem}.
Firstly, we consider the top rational Chern characteristic class of a holomorphic vector bundle.
Let $F$ be a holomorphic vector bundle of rank $r$ over the complex manifold $X$.
The following lemma shows that the top rational Chern characteristic class $\BC_{r}(F)_{\mathbb{Q}}\in H^{r,r}_{BC}(X,\mathbb{Q})$ can be represented by the cycle class associated to the zero locus of some holomorphic sections.
\begin{lem}\label{r-p-l}
Let $s:X\rightarrow F$ be a holomorphic section such that the schematic zero locus $Z(s)$ is smooth.
Then $\BC_{r}(F)_{\mathbb{Q}}$ is equal to the cycle class associated to $Z(s)$ in the rational Bott--Chern cohomology.
\end{lem}
\begin{proof}
Let $\mathcal{I}_{Z(s)}$ be the coherent ideal sheaf of $Z(s)$.
Then we have the canonical Koszul resolution $K(s)=(\wedge^{\bullet}F^{\vee})$ of the coherent sheaf $\CO_{X}/\mathcal{I}_{Z(s)}$, where $F^{\vee}$ is the dual of $F$.
From definition, we get
\begin{equation*}
\BC(\CO_{X}/\mathcal{I}_{Z(s)})_{\mathbb{Q}}=
\prod_{i=0}^{r}\bigl[\BC(\wedge^{i}F^{\vee})_{\mathbb{Q}}\bigr]^{(-1)^{i}}
\in H^{(=)}_{BC}(X, \mathbb{Q}).
\end{equation*}
Note that Lemma \ref{RR-domi-0} is also valid for the Chern characteristic classes of holomorphic vector bundles in the rational Bott--Chern cohomology.
Let $1$ be the trivial complex line bundle over $X$ and therefore there exists a unique power series $f(z_{1},\cdots, z_{r}; w)$ with integer coefficients such that
\begin{eqnarray*}
  \prod_{i=0}^{r}\bigl[\BC(\wedge^{i}F^{\vee})_{\mathbb{Q}}\bigr]^{(-1)^{i}}
   &=&
   \prod_{i=0}^{r}\bigl[\BC(\wedge^{i}F^{\vee}\otimes\CO_{X})_{\mathbb{Q}}\bigr]^{(-1)^{i}} \\
   &=& 1+\BC_{r}(F)_{\mathbb{Q}}\cdot f(F,1),
\end{eqnarray*}
where
$f(F,1)=f(\BC_{1}(F)_{\mathbb{Q}},\cdots,\BC_{r}(F)_{\mathbb{Q}};0)$.
As a result, by the degree reason, we have $\BC_{r}(F)_{\mathbb{Q}}=\BC_{r}(\CO_{X}/\mathcal{I}_{Z(s)})_{\mathbb{Q}}$.
Put $i_{s}:Z(s)\hookrightarrow X$ the natural inclusion.
Because of the isomorphism $\CO_{X}/\mathcal{I}_{Z(s)}\cong i_{s,\ast}\CO_{Z(s)}$, we obtain
$\BC_{r}(F)_{\mathbb{Q}}=\BC_{r}(i_{s,\ast}\CO_{Z(s)})_{\mathbb{Q}}$.
Due to the RRG formula in rational Bott--Chern cohomology (cf. \cite[Corollary 1]{Wu23}), the following identity holds:
$$
i_{s,\ast}(\mathrm{Td_{BC}}(Z(s))_{\mathbb{Q}})
=
\mathrm{ch_{BC}}(i_{s,\ast}\CO_{Z(s)})_{\mathbb{Q}}
\cdot\mathrm{Td_{BC}}(X)_{\mathbb{Q}}
$$
whose lowest degree gives
$i_{s,*}(1)=\BC_{r}(i_{s,\ast}\CO_{Z(s)})_{\mathbb{Q}}.$
Observe that $i_{s,*}(1)=[Z(s)]^{\mathbb{Q}}_{\mathrm{BC}}$ in $H^{r,r}_{BC}(X, \mathbb{Q})$.
It finishes the proof for the equality $\BC_{r}(F)_{\mathbb{Q}}=[Z(s)]^{\mathbb{Q}}_{\mathrm{BC}}$ in rational Bott--Chern cohomology.
\end{proof}

Next, we will establish the analogy of Lemma \ref{key-lem} for rational Chern characteristic classes.
\begin{lem}\label{r-key-lem}
With the same notations as in Subsection \ref{s-sec-4-1},
the equality in Lemma \ref{key-lem} holds in rational Bott--Chern cohomology.
\end{lem}
\begin{proof}
We denote by $z$ the parameter in $\mathbb{P}^{1} = \mathbb{C} \cup \{\infty\}$ and by $[0, \infty]$ a (real) line connecting 0 and $\infty$ in $\mathbb{P}^{1}$ (for example we can take the positive real axis).
Then the function $\ln z$ is well defined on $\mathbb{P}^{1} \setminus [0, \infty]$.
As $[0, \infty]$ is a real codimension one real analytic set of $\mathbb{P}^{1}$, so it well defines a locally integral current.
As a current, we have
$$
d([0, \infty]) = -[0] + [\infty].
$$
We have the decomposition into $(1, 0)$ and $(0, 1)$ parts as
$$
[0, \infty] = -\frac{1}{2 \pi\sqrt{-1}}\overline{\partial}^{\mathbb{P}^{1}}\ln(z) -\frac{1}{2 \pi \sqrt{-1}}\partial^{\mathbb{P}^{1}} \ln(\bar{z}).
$$
(cf. the proof of \cite[Proposition 6.4.3]{Wu20}).
Since the rational Bott--Chern complex is quasi-isomorphic to a soft complex formed by locally integral currents and distributions, we can represent the rational Bott--Chern
cohomology classes by global sections of the soft complex.
In particular, in the spaces of global sections of the soft complex associated to rational
Bott--Chern complex of bidegree $ (1, 1)$, we have
$$
([0] - [\infty], 0) = \delta ([0, \infty], \frac{1}{2 \pi \sqrt{-1}} \ln(\bar{z}) \oplus  -\frac{1}{2 \pi \sqrt{-1}} \ln(z)),
$$
where $\delta$ is the differential.
In the complex Bott--Chern cohomology, this corresponds to the classical Poincar\'{e}--Lelong formula under identification.
As a consequence, the cycle class of $0$ is equal to the cycle class of $\infty$.
Apply \cite[Proposition 6.5.1]{Wu20} to the constants, with the same notations as in Lemma \ref{key-lem}, we get an equality of cycle classes in rational Bott--Chern cohomology
$$
[Y ]^{\mathbb{Q}}_{\mathrm{BC}} - [P ]^{\mathbb{Q}}_{\mathrm{BC}}
 - [ \widetilde{Y} ]^{\mathbb{Q}}_{\mathrm{BC}} = 0.
 $$
Set $\alpha=\BC(l_{\ast}(p_{X}^{\ast}\CF))\in H^{(=)}_{BC}(W, \mathbb{Q})$.
By the projection formula for rational Bott--Chern cohomology, we have
\begin{equation}\label{r-b-5-1}
 \alpha\cdot[Y ]^{\mathbb{Q}}_{\mathrm{BC}} - \alpha\cdot[P ]^{\mathbb{Q}}_{\mathrm{BC}}
 - \alpha\cdot[ \widetilde{Y} ]^{\mathbb{Q}}_{\mathrm{BC}}
 = j_{0,\ast}(j^{\ast}_{0}\alpha)-
 \hat{\jmath}_{\ast}(\hat{\jmath}^{\ast}\alpha)-g_{\ast}(g^{\ast}\alpha)
 =  0
 \end{equation}
Taking $q_{1,*}$ to \eqref{r-b-5-1} deduces the equality
$
j^{\ast}_{0}\alpha=\mu_{\ast}(\hat{\jmath}^{\ast}\alpha)
$
in $H^{(=)}_{BC}(Y, \mathbb{Q})$ which means
$\BC(\imath_{\ast}\CF)_{\mathbb{Q}}
=\mu_{\ast}\BC(\overline{\imath}_{\ast}\CF)_{\mathbb{Q}}$.
\end{proof}

Based on Lemma \ref{r-p-l} and Lemma \ref{r-key-lem}, we can show that the Riemann--Roch without denominators for rational Chern characteristic classes is valid by following the steps in the proof of Theorem \ref{RR-domi}.
Consequently, without essential changes, the proof of Theorem \ref{thm1} still works in the rational case and we are led to the following result.
\begin{thm}\label{thm2}
Under the same assumptions as Theorem \ref{thm1}, the blow-up formula holds with the substitution of the refined Chern characteristic classes for the rational Chern characteristic classes throughout.
\end{thm}
\begin{rem}
On the one hand, the ring structure of rational Bott--Chern cohomology is considerably more complicated than that of complex Bott--Chern cohomology (cf. \cite{Sch07, Wu23}).
On the other hand, the axiomatic approach to the definition of Chern characters of a coherent sheaf following \cite{Gri10} is quite abstract and difficult to make explicit.
It is noteworthy that the coherent sheaves appearing in Theorem \ref{thm1} are rather special, and their Chern characters can be described using the approach in \cite{Wu23} or via the generalized Poincar\'{e}--Lelong formulae for holomorphic vector bundles (cf. \cite{An07, CGL20, Wu25}).
Observe that there exists a canonical morphism:
$$
\epsilon: H^{(=)}_{BC}(X, \mathbb{Q})\longrightarrow H^{(=)}_{BC}(X, \mathbb{R}).
$$
Comparing to Definition \ref{t-real-b-c}, it follows from \cite[Corollary 9.4.3]{BSW23}  that $\epsilon$ maps $\BC(\CF^{\bullet})_{\mathbb{Q}}$ to the (total) real Bott--Chern cohomology class $\BC(\CF^{\bullet})$.
As a result, the blow-up formulae in Theorems \ref{thm1} and \ref{thm2} are compatible under the natural morphism $\epsilon$.
For these reasons, we retain both approaches in this paper for potential practical applications, and we believe that these two cases have their own interest.
\end{rem}
\section{Examples}\label{sec6}

It is noteworthy that for a holomorphic vector bundle over a compact K\"{a}hler manifold its Chern classes and refined Chern classes agree (cf. \cite[Theorem 1]{Big69}).
For the blow-up of an algebraic threefold, the Chern classes can be determined by some ad hoc methods (cf. \cite[Page 609]{GH94}).
In this section, we discuss the refined Chern classes of the blow-up of a general surface or a general threefold.
\subsection{Blowing up surfaces and threefolds}
Assume that $X$ is a point, then the form $\alpha$ in Theorem \ref{thm1} can be expressed as
$$
\alpha=1+\sum^{n}_{i=1}\binom{n}{i}(\zeta^{i}-\zeta^{i-1}).
$$
So that
$$
\BC(\widetilde{Y})-\pi^{\ast}\BC(Y)=\jmath_{\ast}\biggl[1+\sum^{n}_{i=1}\binom{n}{i}(\zeta^{i}-\zeta^{i-1})\biggr].
$$
Consider the first and the second refined Chern characteristic classes of a general blow-up, from Theorem \ref{thm1}, we get
\begin{equation}\label{bc1}
\BC_{1}(\widetilde{Y})=\pi^{\ast}\BC_{1}(Y)+(1-r)\cdot[E]_{\mathrm{BC}}
\end{equation}
 and
\begin{equation}\label{c2}
\BC_{2}(\widetilde{Y})=\pi^{\ast}\BC_{2}(Y)+
\jmath_{\ast}\biggl[\frac{r(3-r)}{2}\zeta+(2-r)\rho^{\ast}\BC_{1}(N)+
(1-r)\rho^{\ast}\BC_{1}(X)\biggr],
\end{equation}
where $[E]_{\mathrm{BC}}$ is the Bott--Chern cohomology class represented by the current $[E]$.
Using Theorem \ref{thm1}, we can compute the higher refined Chern characteristic classes of a general blow-up; however, the computation is complicated.

\begin{rem}
Consider the canonical line bundle of $\widetilde{Y}$.
The identity
\begin{equation}\label{kk}
K_{\widetilde{Y}}=\CO_{\widetilde{Y}}((r-1)E)\otimes\pi^{\ast}K_{Y}
\end{equation}
is valid (cf. \cite[Proposition 12.7]{Dem12}).
Comparing to the first Chern class of a blow-up \cite[Page 608]{GH94}, the formula \eqref{bc1} also can be deduced from  \eqref{kk}.
\end{rem}

Recall that $\zeta=\BC_{1}(\CO_{E}(1))$ and $\CO_{E}(1)=\jmath^{\ast}\CO_{\widetilde{Y}}(E)$, where $\CO_{\widetilde{Y}}(E)$ is the associated holomorphic line bundle of the exceptional divisor $E$.
Note that $\BC_{1}(\CO_{\widetilde{Y}}(E))=[E]_{\mathrm{BC}}$, see Appendix \ref{app}.
By the projection formula, we have
\begin{eqnarray*}
\jmath_{\ast}\zeta&=&\jmath_{\ast}(\jmath^{\ast}\BC_{1}(\CO_{\widetilde{Y}}(E)))=\jmath_{\ast}(\jmath^{\ast}\BC_{1}(\CO_{\widetilde{Y}}(E))\cdot 1)\\
&=&\BC_{1}(\CO_{\widetilde{Y}}(E))\cdot \jmath_{\ast}1=[E]_{\mathrm{BC}}\cdot[E]_{\mathrm{BC}}.
\end{eqnarray*}
Suppose $Y$ is a compact complex surface, then $X$ is a point and we have
\begin{equation*}
\BC_{2}(\widetilde{Y})=\pi^{\ast}\BC_{2}(Y)+\jmath_{\ast}\zeta=\pi^{\ast}\BC_{2}(Y)+([E]_{\mathrm{BC}})^{2}.
\end{equation*}
If $Y$ is a compact complex threefold, then we can blow up $Y$ at a point or along a smooth curve.
\begin{prop}
Let $\widetilde{Y}$ be the blow-up of the threefold $Y$ at a point $X=\{pt\}$, then we have:
\begin{itemize}
\item[(i)] $\BC_{2}(\widetilde{Y})=\pi^{\ast}\BC_{2}(Y)$;
\item[(ii)]  $\BC_{3}(\widetilde{Y})=\pi^{\ast}\BC_{3}(Y)+2([E]_{\mathrm{BC}})^{3}$.
\end{itemize}
\end{prop}
\begin{proof}
Since $r=3$ the assertion (i) is a direct consequence of \eqref{c2}.
From Theorem \ref{thm1}, we get
$\BC_{3}(\widetilde{Y})=\pi^{\ast}\BC_{3}(Y)+2\jmath_{\ast}\zeta^{2}$.
Observe that $\jmath_{\ast}\zeta^{2}=([E]_{\mathrm{BC}})^{3}$ and this leads to the conclusion.
\end{proof}

Now we assume that $X$ is a smooth complex curve in the threefold $Y$.
\begin{prop}\label{3-curve}
Let $\widetilde{Y}$ be the blow-up of $Y$ with the center $X$, then we have
\begin{itemize}
\item[(i)] $\BC_{2}(\widetilde{Y})=\pi^{\ast}\BC_{2}(Y)-\pi^{\ast}\BC_{1}(Y)\cdot [E]_{\mathrm{BC}}  + \pi^{\ast}[X]_{\mathrm{BC}}$;
\item[(ii)] $\BC_{3}(\widetilde{Y})=\pi^{\ast}\BC_{3}(Y)+\jmath_{\ast}(\rho^{\ast}\BC_{1}(X)\cdot\zeta)$.
\end{itemize}
\end{prop}
\begin{proof}
From Theorem \ref{thm1}, we get
\begin{equation*}
\BC_{2}(\widetilde{Y})-\pi^{\ast}\BC_{2}(Y) = \jmath_{\ast}[-\rho^{\ast}\BC_{1}(X)+\zeta] .
\end{equation*}
Since $\BC_{1}(X)=\imath^{\ast}\BC_{1}(Y)-\BC_{1}(N)$,
using \eqref{tau-ex}, Proposition \ref{for-cl} and the projection formula, we obtain
\begin{eqnarray*}
\BC_{2}(\widetilde{Y})-\pi^{\ast}\BC_{2}(Y)
& = & \jmath_{\ast}[-\rho^{\ast}\BC_{1}(X)+\zeta] \\
& = & \jmath_{\ast}[-\rho^{\ast}\imath^{\ast}\BC_{1}(Y)]+\jmath_{\ast}[\rho^{\ast}\BC_{1}(N)+\zeta] \\
& = & \jmath_{\ast}[-\jmath^{\ast}\pi^{\ast}\BC_{1}(Y)]+\jmath_{\ast}[\BC_{1}(Q_{E})] \\
& = & -\pi^{\ast}\BC_{1}(Y)\cdot [E]_{\mathrm{BC}} + \pi^{\ast}( \imath_{\ast} 1) \\
& = & -\pi^{\ast}\BC_{1}(Y)\cdot [E]_{\mathrm{BC}}  + \pi^{\ast}[X]_{\mathrm{BC}} .
\end{eqnarray*}

By Theorem \ref{thm1} again, the top refined Chern class of $\widetilde{Y}$ can be expressed as
$$
\BC_{3}(\widetilde{Y})=\pi^{\ast}\BC_{3}(Y)+\jmath_{\ast}[\zeta^{2}+\rho^{\ast}\BC_{1}(N)\cdot\zeta+\rho^{\ast}\BC_{1}(X)\cdot\zeta+\rho^{\ast}\BC_{2}(N)].
$$
Due to \eqref{tau-ex}, we have $\rho^{\ast}\BC(N)=(1-\zeta)\cdot\BC(Q_{E})$ and thus $\rho^{\ast}\BC_{2}(N)=-\zeta\cdot\BC_{1}(Q_{E})$ since the rank of $Q_{E}$ is 1.
It follows $\zeta^{2}+\rho^{\ast}\BC_{1}(N)\cdot\zeta+\rho^{\ast}\BC_{2}(N)=0$ and this completes the proof of assertion (ii).
\end{proof}

\subsection{Iwasawa manifolds}
Recall the construction of Iwasawa manifolds.
Denote by $\mathbb{H}(3; \mathbb{C})$ the Heisenberg Lie group
$$
\mathbb{H}(3; \mathbb{C}):=
\Bigg\{ \small{\left(\begin{array}{ccc}
1 & z_{1} & z_{3}\\
0 & 1  & z_{2} \\
0 & 0 & 1
\end{array}
\right)}\mid z_{1}, z_{2}, z_{3} \in \mathbb{C} \Bigg\}\subset\mathrm{GL}(3, \mathbb{C}).
$$
Consider the discrete group
$\mathbb{G}_{3}:=\mathrm{GL}(3; \mathbb{Z}[\sqrt{-1}])\cap \mathbb{H}(3; \mathbb{C})$,
where $\mathbb{Z}[\sqrt{-1}]=\{a+b\sqrt{-1}\mid a,b\in \mathbb{Z}\}$ is the Gaussian integers.
The left multiplication gives rise to a natural $\mathbb{G}_{3}$-action on $\mathbb{H}(3; \mathbb{C})$ which is properly discontinuous.
The corresponding quotient space
$
\mathbb{I}_{3}:=\mathbb{H}(3; \mathbb{C}) / \mathbb{G}_{3}
$
is a $6$-dimensional compact smooth real manifold.
Moreover, there is a natural complex structure $J_{0}$ on $\mathbb{I}_{3}$ induced from the standard complex structure on $\mathbb{H}(3; \mathbb{C})\cong\mathbb{C}^{3}$.
The compact complex threefold $(\mathbb{I}_{3}, J_{0})$ is called the {\it Iwasawa manifold}, which is non-K\"{a}hler, no-formal, and holomorphically parallelizable.
By definition, the projection
$${
\left(\begin{array}{ccc}
1 & z_{1} & z_{3}\\
0 & 1  & z_{2} \\
0 & 0 & 1
\end{array}
\right)}
\longmapsto (z_{1}, z_{2})$$
determines a natural holomorphic submersion
$$
f: \mathbb{I}_{3} \longrightarrow \mathbb{T}_{\mathbb{C}}^{2}=\mathbb{C}/ \mathbb{Z}[\sqrt{-1}] \times\mathbb{C}/ \mathbb{Z}[\sqrt{-1}].
$$
This implies that $\mathbb{I}_{3}$ is a holomorphic fibre bundle over $\mathbb{T}^{2}_{\mathbb{C}}$ and all fibres of $f$ are isomorphic to the complex torus $\mathbb{T}_{\mathbb{C}}=\mathbb{C}/ \mathbb{Z}[\sqrt{-1}] $.

Consider the fibre $C:=f^{-1}([0, 0])=f^{-1}([c_{1}, c_{2}])$, where $c_{1}, c_{2}\in \mathbb{Z}[\sqrt{-1}]$.
Let $\pi:\tilde{\mathbb{I}}_{3}\rightarrow\mathbb{I}_{3}$ be the blow-up of $\mathbb{I}_{3}$ along the curve $C$ (cf. \cite[Example 1]{CY19}).
Next, we compute the refined Chern characteristic classes of $\tilde{\mathbb{I}}_{3}$.
Note that the space of $\mathbb{H}(3; \mathbb{C})$-invariant $(1,0)$-forms on $\mathbb{H}(3; \mathbb{C})$ has a basis $\omega^{1}=dz_{1}$, $\omega^{2}=dz_{2}$, $\omega^{3}=dz_{3}-z_{1}dz_{2}$ satisfying the structure equations
$$
\begin{cases}
d^{\mathbb{I}_{3}}\omega^{1}=0, \\
d^{\mathbb{I}_{3}}\omega^{2}=0,\\
d^{\mathbb{I}_{3}}\omega^{3}=-\omega^{1}\wedge\omega^{2}.
\end{cases}
$$
Put $\omega^{s\overline{t}}=\omega^{s}\wedge\overline{\omega^{t}}$ for all $s,t>0$.
Since the holomorphic tangent bundle of $\mathbb{I}_{3}$ is holomorphically trivial we get $\BC_{i}(\mathbb{I}_{3})=0$, for $i=1,2,3$.
Note that $C$ is a complex torus.
It follows that the holomorphic tangent bundle of $C$ is trivial and thus $\BC_{1}(C)=0$.
Due to \eqref{bc1}, the first refined Chern characteristic class of $\tilde{\mathbb{I}}_{3}$ has the form of
$\BC_{1}(\tilde{\mathbb{I}}_{3})=-[E]_{\mathrm{BC}}$,
where $E\cong\mathbb{P}(N_{C/\mathbb{I}_{3}})$ is the exceptional divisor of the blow-up.
By Proposition \ref{3-curve}, we have $\BC_{2}(\tilde{\mathbb{I}}_{3})=\pi^{\ast}[C]_{\mathrm{BC}}$ and $\BC_{3}(\tilde{\mathbb{I}}_{3})=0$.
We claim that $[C]_{\mathrm{BC}}$ is zero in $H^{2,2}_{BC}(\mathbb{I}_{3}, \mathbb{R})$.
According to the canonical isomorphism $H^{2,2}_{BC}(\mathbb{I}_{3}, \mathbb{R})\cong H^{2,2}_{BC}(\mathfrak{D}^{(=)}(\mathbb{I}_{3}, \mathbb{R}))$, there exists a $d$-closed differential form $\eta\in\Omega^{2,2}(\mathbb{I}_{3}, \mathbb{R})$ such that $\{\eta\}=[C]_{\mathrm{BC}}$.
Observe that $C=\{[0, 0, z_{3}]\in\mathbb{I}_{3}\, |\, z_{3}\in\mathbb{C} \}$.
It follows that $\{\eta\}$ belongs to the vector space generated by $\{\omega^{12\overline{1}\overline{2}}\}$.
A straightforward computation shows $\omega^{12\bar{1}\bar{2}}=-\partial^{\mathbb{I}_{3}}\overline{\partial}^{\mathbb{I}_{3}}(\omega^{3\bar{3}})$ and thus $\{\omega^{12\bar{1}\bar{2}}\}=0$ in $H^{2,2}_{BC}(\mathbb{I}_{3}, \mathbb{R})$.
As a result, we get $\BC_{2}(\tilde{\mathbb{I}}_{3})=0$.

\appendix
\section{Some formulae for Bott--Chern cohomology}\label{app}
Here we collect a number of key formulae on complex Bott--Chern cohomology, which are used in the paper.

Suppose $Y$ is a compact complex manifold of complex dimension $n$.
For any $z\in\mathbb{P}^{1}$, let $j_{z}:Y\times\{z\}\hookrightarrow Y\times\mathbb{P}^{1}$ be the inclusion  map.
It follows from the Bott--Chern projective bundle formula  (cf. \cite[Corollary 3.3]{YY23}) that the \emph{$\mathbb{P}^{1}$-homotopy principle} for Bott--Chern cohomology is valid, i.e., the morphism
$$
j^{\ast}_{z}:H^{\ast,\ast}_{BC}(Y\times\mathbb{P}^{1})\rightarrow H^{\ast,\ast}_{BC}(Y\times\{z\})\cong H^{\ast,\ast}_{BC}(Y)
$$
is independent from $z$.

Assume that  $i:Z\hookrightarrow Y$ is a smooth analytic hypersurface, then $Z$ is a divisor on $Y$.
Let $L$ be the associated holomorphic line bundle of $Z$ and $[Z]_{\mathrm{BC}}$ the Bott--Chern cohomology class represented by the current $[Z]$.
Let $\nabla$ be a connection on $L$, then the curvature matrix $\Theta$ is a $d$-closed, real $(1,1)$-form.
On the one hand, from definition, the first refined Chern characteristic class $\BC_{1}(L)$ is equal to
$$
\{\frac{\sqrt{-1}}{2\pi}\Theta\}\in H^{1,1}_{BC}(Y,\mathbb{R}).
$$
On the other hand, a direct computation shows
$$
\frac{\sqrt{-1}}{2\pi}\int_{Y}\Theta\wedge\alpha=\int_{Z}i^{\ast}\alpha,
$$
for any $\alpha\in\Omega^{n-1,n-1}(Y)$.
This implies $\{\frac{\sqrt{-1}}{2\pi}\Theta\}=[Z]_{\mathrm{BC}}$ under the smoothing of Bott--Chern cohomology $H^{1,1}_{BC}(Y)\cong H^{1,1}_{BC}(\mathfrak{D}^{(=)}(Y))$ and therefore we get $\BC_{1}(L)=[Z]_{\mathrm{BC}}$.
In particular, we have the following \emph{self-intersection formula} for Bott--Chern cohomology and it's  proof is the same as that of \cite[Proposition 1]{Gri10} and \cite[Proposition 6.5.4]{Wu20}.

\begin{prop}\label{sel-int}
For any Bott--Chern cohomology class $\{\alpha\}\in H^{p,q}_{BC}(Z)$, the following identity is valid
$$
i^{\ast}i_{\ast}\{\alpha\}=\{\alpha\}\cdot \BC_{1}(N_{Z/Y}).
$$
\end{prop}

As a direct consequence of the self-intersection formula, for any $\{\alpha\}, \{\beta\} \in H^{p,q}_{BC}(Z)$,  we have
\begin{eqnarray*}
i_{\ast}\{\alpha\} \cdot i_{\ast}\{\beta\} &=&i_{\ast}(\{\alpha\}\cdot i^{\ast}i_{\ast}\{\beta\})\\
&=&i_{\ast}(\{\alpha\}\cdot \{\beta\} \cdot \BC_{1}(N_{Z/Y})).
\end{eqnarray*}
Here the first equality results from the projection formula.
\begin{rem}
In general, given a complex manifold $X$ with an irreducible $k$-codimensional analytic cycle $Z$ in $X$, there exists a natural injective morphism between the local integral and complex $(k,k)$-Bott--Chern cohomology groups of $X$ supported on $Z$:
$$
\epsilon:H^{k,k}_{BC,\,|Z|}(X, \mathbb{Z})\longrightarrow H^{k,k}_{BC,\,|Z|}(X, \mathbb{C}).
$$
The current of integrations on $Z$ represents a class $[Z]^{\mathbb{Z}}_{\mathrm{BC}}$ in $H^{k,k}_{BC,\,|Z|}(X, \mathbb{Z})$  and also a class $[Z]^{\mathbb{C}}_{\mathrm{BC}}$ in $H^{k,k}_{BC,\,|Z|}(X, \mathbb{C})$ satisfying
$
\epsilon([Z]^{\mathbb{Z}}_{\mathrm{BC}})=[Z]^{\mathbb{C}}_{\mathrm{BC}},
$
for details see \cite[\S\,6.5]{Wu20}.
Suppose we want to prove a formula involving complex Bott--Chern classes represented by irreducible analytic cycles, it is sufficient to verify the same formula for corresponding integral Bott--Chern cohomology classes since the natural morphism $\epsilon$ is injective.
\end{rem}

Now we assume that $\imath:X\hookrightarrow Y$ is a closed complex submanifold with codimension $\mathrm{codim}_{\mathbb{C}}\,X=r\geq 2$.
Recall the blow-up diagram:
\begin{equation*}
\vcenter{
\xymatrix@=1.0cm{
E \ar[d]_{\rho} \ar@{^{(}->}[r]^{\jmath} & \widetilde{Y}\ar[d]^{\pi}\\
 X \ar@{^{(}->}[r]^{\imath} & Y}
 }
\end{equation*}
As an application of Proposition \ref{sel-int}, we can establish the inverse of the morphism in the Bott--Chern blow-up formula \cite[Theorem 3.7]{YY23}; see also \cite[\S 4]{Men19} and \cite[Proposition 4.15]{Men20} for different argument.

\begin{prop}\label{b-l-f}
For any $0\leq p, q\leq n$, there exists a canonical isomorphism
\begin{equation*}
\Psi: H^{p,q}_{BC}(Y)\oplus\biggl[\bigoplus^{r-2}_{i=0}H^{p-i-1,q-i-1}_{BC}(X)\biggr]
\stackrel{\simeq}\longrightarrow H^{p,q}_{BC}(\widetilde{Y}),
\end{equation*}
where $\Psi=\pi^{\ast}+\sum^{r-2}_{i=0}\jmath_{\ast}[\zeta^{i}\wedge\rho^{\ast}(-)]$
and $\zeta=\BC_{1}(\CO_{E}(1))$.
\end{prop}
\begin{proof}
Note that $\Psi$ is a $\mathbb{C}$-linear map of complex vector spaces with the same finite dimension.
To prove the assertion, it is sufficient to show that $\Psi$ is injective.
Let $\{\alpha\}_{p,q}\in H^{p,q}_{BC}(Y)$ and $\{\alpha\}_{p-i-1,q-i-1}\in H^{p-i-1,q-i-1}_{BC}(X)$ for each $0\leq i\leq r-2$.
Assume that
$$\Psi\bigl(\{\alpha\}_{p,q}, \{\alpha\}_{p-1,q-1},\cdots,\{\alpha\}_{p-c+1,q-c+1}\bigr)=0,$$
then it follows from the definition of $\Psi$ that the identities
$\pi^{\ast}\{\alpha\}_{p,q}=0$ and
$$\jmath_{\ast}\biggl[\sum^{r-2}_{i=0}\zeta^{i}\wedge\rho^{\ast}\{\alpha\}_{p-i-1,q-i-1}\biggr]=0$$
 are valid.
As $\pi^{\ast}$ is injective, we get $\{\alpha\}_{p,q}=0$.
It remains to show $\{\alpha\}_{p-i-1,q-i-1}=0$ for all $0\leq i\leq c-2$.
According to the self-intersection formula in Proposition \ref{sel-int}, we get
\begin{eqnarray*}
0 &=& \jmath^{\ast}\jmath_{\ast}\biggl[\sum^{r-2}_{i=0}
  \zeta^{i}\wedge\rho^{\ast}\{\alpha\}_{p-i-1,q-i-1} \biggr] \\
   &=& \zeta\wedge\biggl[\sum^{r-2}_{i=0}
  \zeta^{i}\wedge\rho^{\ast}\{\alpha\}_{p-i-1,q-i-1} \biggr]\\
   &=& \sum^{r-1}_{l=1}
  \zeta^{l}\wedge\rho^{\ast}\{\alpha\}_{p-l,q-l}.
\end{eqnarray*}
As a direct consequence of the projective bundle formula for Bott--Chern cohomology, for any $1\leq l\leq c-1$, we have $\{\alpha\}_{p-l,q-l}=0$ and therefore $\Psi$ is injective.
\end{proof}

Consider the tautological exact sequence \eqref{tau-ex}:
\begin{equation*}
\xymatrix@C=0.5cm{
  0 \ar[r] &  \CO_{E}(-1) \ar[r]^{} & \rho^{\ast}N \ar[r]^{} & Q_{E}  \ar[r] & 0.}
\end{equation*}
Based on the self-intersection formula for hypersurface (Proposition \ref{sel-int}) and the blow-up formula (Proposition \ref{b-l-f}),
we can establish the following key formula called the \emph{formule-clef} for Bott--Chern cohomology.

\begin{prop}\label{for-cl}
For any Bott--Chern cohomology class $\{\alpha\}\in H^{p,q}_{BC}(X)$,  the identity
$$\pi^{\ast}\imath_{\ast}\{\alpha\}=\jmath_{\ast}(\rho^{\ast}\{\alpha\}\cdot\BC_{r-1}(Q_{E}))$$
holds in $H^{p+r,q+r}_{BC}(\widetilde{Y})$, where $0\leq p,q \leq n-r$.
\end{prop}

The formule-clef for integral Bott--Chern cohomology was proved by the first named author \cite[Proposition 13]{Wu23} and the proof of the formula above is identical with the proof of \cite[Proposition 13]{Wu23} by replacing the integral Bott--Chern cohomology with the complex Bott--Chern cohomology throughout.

\begin{rem}
It was Grothendieck who first conjectured the self-intersection formula and the formule-clef for Chow rings of non-singular quasi-projective varieties.
In \cite{LMS75}, using the construction of deformation to the normal cone, Lascu--Mumford--Scott gave an elementary proof of these formulae.
By use of deformation to the normal bundle for complex manifolds, their analogues in Deligne cohomology and integral Bott--Chern cohomology are established in \cite{Gri10} and \cite{Wu20}, respectively.
\end{rem}


\end{document}